\RequirePackage[l2tabu, orthodox]{nag}
\documentclass[12pt, a4paper]{article}

\makeatletter

\usepackage{hyperref}

\usepackage[top=1truein, bottom=1truein, left=1truein, right=1truein]{geometry}

\usepackage{amsfonts}
\usepackage{amsmath}
\numberwithin{equation}{section}
\usepackage{amssymb}
\usepackage{bm}
\usepackage{mathtools}

\usepackage{type1cm}
\usepackage[T1]{fontenc}
\usepackage{lmodern}

\usepackage{algorithm}
\usepackage[noend]{algpseudocode}

\algnewcommand\algorithmicinput{\makebox[\widthof{\textbf{Output}}][l]{\textbf{Input}}\textbf{:}}
\algnewcommand\Input{\item[\algorithmicinput]}
\algnewcommand\algorithmicoutput{\textbf{Output:}}
\algnewcommand\Output{\item[\algorithmicoutput]}

\algnewcommand\algorithmicto{\textbf{to}}
\algdef{SE}[FOR]{ForTo}{EndFor}[2]{\algorithmicfor\ #1 \algorithmicto\ #2 \algorithmicdo}{\algorithmicend\ \algorithmicfor}\algtext*{EndFor}

\usepackage{calc}
\usepackage{enumitem}
\setenumerate{label=(\arabic*)}
\newenvironment{namedenum}[2][9]{%
  \begin{enumerate}[noitemsep,labelindent=1em,labelwidth=\widthof{(#2#1)},leftmargin=!,align=right,label*=(#2\arabic*)]%
  
}{%
  \end{enumerate}%
}

\usepackage{amsthm}
\newtheorem{theorem}{Theorem}
\newtheorem{lemma}[theorem]{Lemma}
\newtheorem{corollary}[theorem]{Corollary}
\newtheorem{proposition}[theorem]{Proposition}

\theoremstyle{definition}
\newtheorem{definition}[theorem]{Definition}
\newtheorem{remark}{Remark}

\usepackage{subcaption}
\usepackage{xcolor}
\usepackage{graphicx}
\usepackage{tikz}
\usetikzlibrary{arrows, calc, intersections, positioning, shapes}

\usepackage{authblk}

\usepackage{multirow}
\usepackage{todonotes}
\usepackage{url}

\usepackage{gfncmd}

\usebbsetcapital
\usebinaryPower
\usescriptphi
\usescriptepsilon

\newcommand{\nospleft}{\mathopen{}\mathclose\bgroup\left}
\newcommand{\nospright}{\aftergroup\egroup\right}

\newcommand{\nospparen}[1]{\nospleft({#1}\nospright)}
\let\paren=\nospparen

\renewcommand{\setprn}[1]{\nospleft\{{#1}\nospright\}}
\renewcommand{\setprnsep}[2]{\nospleft\{{#1}\mathrel{}\middle|\mathrel{}{#2}\nospright\}}
\newcommand{\condition}[2]{{{#1}\::\:{#2}}}
\renewcommand{\absprn}[1]{\nospleft|{#1}\nospright|}
\renewcommand{\sqbracket}[1]{\nospleft[{#1}\nospright]}
\newcommand{\norm}[1]{\nospleft\lVert{#1}\nospright\rVert}
\newcommand{\ceil}[1]{\nospleft\lceil{#1}\nospright\rceil}
\newcommand{\floor}[1]{\nospleft\lfloor{#1}\nospright\rfloor}

\renewcommand{\app}[2]{{#1}\nospparen{#2}}
\newcommand{\nospsqparen}[1]{\nospleft[{#1}\nospright]}
\newcommand{\sqapp}[2]{{#1}\nospsqparen{#2}}

\renewcommand{\Order}[1]{\app{\mathrm{O}}{#1}}
\newcommand{\maxflow}[2]{\app{\mathrm{MF}}{{#1}, {#2}}}

\newcommand{\wc}{,\;}

\newcommand{\setRext}{\overline{\setR}}

\newcommand{\supp}{\mathop{\mathrm{supp}}}

\renewcommand{\meet}{\mathbin\wedge}
\renewcommand{\join}{\mathbin\vee}
\newcommand{\sqmeet}{\mathbin\sqcap}
\newcommand{\sqjoin}{\mathbin\sqcup}
\renewcommand{\bigland}{\bigwedge}

\newcommand{\inconsistent}{\mathrel\smile}
\newcommand{\mininconsistent}{\mathrel{\overset{{}_{\scriptscriptstyle{\bullet}}}{\smile}}}

\newcommand{\intset}[1]{\sqbracket{#1}}
\newcommand{\irreducibles}[1]{{#1}^\mathrm{ir}}
\newcommand{\consideals}[1]{\app{\mathcal{C}}{#1}}
\newcommand{\maxconsideals}[1]{\app{\mathcal{C}_{\mathrm{max}}}{#1}}
\newcommand{\minimizers}[1]{\app{\mathcal{D}_{\mathrm{min}}}{#1}}
\newcommand{\vectorpos}[2]{\overset{\overset{#2}{\smash{\smallsmile}\rule{0pt}{3pt}}}{#1}}

\newcommand{\partfix}[3]{{#1}_{#2, #3}}
\newcommand{\fixarg}[3]{{#1}_{{#2},{#3}}}
\newcommand{\legalize}[1]{\smash{\check{#1}}}

\makeatother

\title{A compact representation for minimizers of $k$-submodular functions%
  \footnote{An earlier version of this paper was presented at the 4th International Symposium on Combinatorial Optimization (ISCO 2016), Vietri sul Mare, Italy, May 16--18, 2016~\cite{isco2016}.}%
}
\author{Hiroshi Hirai}
\author{Taihei Oki}
\affil{%
  Department of Mathematical Informatics,\\%
  Graduate School of Information Science and Technology,\\%
  The University of Tokyo, Tokyo, 113-8656, Japan\\%
  Email: \{\href{hirai@mist.i.u-tokyo.ac.jp}{\texttt{hirai}},%
           \href{taihei_oki@mist.i.u-tokyo.ac.jp}{\texttt{taihei\_oki}}%
         \}\texttt{@mist.i.u-tokyo.ac.jp}%
}

\begin{document}
\maketitle

\begin{abstract}
  A $k$-submodular function is a generalization of submodular and bisubmodular functions.
  This paper establishes a compact representation for minimizers of a $k$-submodular function by a poset with inconsistent pairs (PIP).
  This is a generalization of Ando--Fujishige's signed poset representation for minimizers of a bisubmodular function.
  We completely characterize the class of PIPs (elementary PIPs) arising from $k$-submodular functions.
  We give algorithms to construct the elementary PIP of minimizers of a $k$-submodular function $f$ for three cases: (i) a minimizing oracle of $f$ is available, (ii) $f$ is network-representable, and (iii) $f$ arises from a Potts energy function.
  Furthermore, we provide an efficient enumeration algorithm for all maximal minimizers of a Potts $k$-submodular function.
  Our results are applicable to obtain all maximal persistent labelings in actual computer vision problems.
  We present experimental results for real vision instances.
  \\
  
  \noindent \textbf{Keywords:} $k$-submodular function, Birkhoff representation theorem, poset with inconsistent pairs (PIP), Potts energy function
\end{abstract}

\section{Introduction}
Minimizers of a submodular function form a distributive lattice, and are compactly represented by a poset (partially ordered set) via Birkhoff representation theorem.
This fact reveals a useful hierarchical structure of the minimizers, and is applied to the DM-decomposition of matrices and further refined block-triangular decompositions~\cite{Murota2000}.

In this paper, we address such a Birkhoff-type representation for minimizers of a \emph{$k$-submodular function}.
Here $k$-submodular functions, introduced by Huber--Kolmogorov~\cite{Huber2012}, are functions on $\setprn{0, 1, 2, \ldots, k}^n$ defined by submodular-type inequalities.
This generalization of (bi)submodular functions has recently gained attention for algorithm design and modeling~\cite{Gridchyn2013, Hirai2015, Hirai2016, IwataS2016, IwataY2016}.

Our main result is to establish a compact representation for minimizers of a $k$-submodular function. 
This can be viewed as a generalization of the above poset representation for submodular functions and Ando--Fujishige's signed poset representation for bisubmodular functions~\cite{Ando1994}. 
A feature of our representation is to utilize a \emph{poset with inconsistent pairs (PIP)}~\cite{Ardila2011, Barthelemy1993, Nielsen1981}, which is a discrete structure having a stronger power of expression than that of a signed poset. 
Actually a PIP is a poset endowed with an additional binary relation (\emph{inconsistency relation}), and is viewed as a poset reformulation of 2-CNF.
This concept, also known as an \emph{event structure}, was first introduced by Nielsen--Plotkin--Winskel~\cite{Nielsen1981} as a model of concurrency in theoretical computer science, and was independently considered by Barthelemy--Constantin~\cite{Barthelemy1993} to establish a Birkhoff-type representation theorem for a \emph{median semilattice}---a semilattice generalization of a distributive lattice.
A PIP was recently rediscovered by Ardila--Owen--Sullivant~\cite{Ardila2011} to represent nonpositively-curved cube complexes; the term ``PIP'' is due to them.

Our results consist of structural and algorithmic ones, summarized as follows:

\paragraph{Structural results.}
We show that minimizers of a $k$-submodular function form a median semilattice (Lemma~3).
By a Birkhoff-type representation theorem~\cite{Barthelemy1993} for median semilattices, the minimizer set is represented by a PIP, where minimizers are encoded into special ideals in the PIP, called \emph{consistent ideals}.
PIPs arising from $k$-submodular functions are rather special.
We completely characterize such PIPs (Theorem~\ref{thm:closed_set_and_elementary_pip}), which we call \emph{elementary}.
This representation is actually compact.
We show that the size of the elementary PIP for a $k$-submodular function of $n$ variables is $\Order{kn}$ (Proposition~\ref{prop:number_of_irreducibles}).

\paragraph{Algorithmic results.} 
We present algorithms to construct the elementary PIP of the minimizers of a $k$-submodular function $f$ under the following three situations:
\begin{itemize}
  \item[(i)] A minimizing oracle of $f$ is given.
  \item[(ii)] $f$ is network-representable.
  \item[(iii)] $f$ arises from a Potts energy function.
\end{itemize}
For (i), we show that the PIP is obtained by calling the minimizing oracle $\Order{kn^2}$ time (Theorem~\ref{thm:alg_oracle}).
Notice that a polynomial time algorithm to minimize $k$-submodular functions is not known for the value-oracle model but is known for the valued-CSP model~\cite{Kolmogorov2015}.
Our result for (i) is applicable to such a case.

For (ii) (and (iii)), we consider a class of efficiently minimizable $k$-submodular functions considered in~\cite{IwataY2016}, where a $k$-submodular function in this class is represented by the cut function in a network of $\Order{kn}$ vertices and can be minimized by a minimum-cut computation.
We show that the PIP is naturally obtained from the residual graph of a maximum flow in the network (Theorems~\ref{thm:network_algo} and~\ref{thm:network_time}). 

For (iii), we deal with a $k$-submodular function $\funcdoms{\tilde{g}}{\setprn{0, 1, 2, \ldots, k}^n}{\setR}$ obtained from a $k$-label Potts energy function $\funcdoms{g}{\setprn{1, 2, \ldots, k}^n}{\setR}$ by adding the $0$-label (meaning ``non-labeled'').
Such a $k$-submodular function, called \emph{Potts $k$-submodular}, is particularly useful in vision applications.
Indeed, via the \emph{persistency property}~\cite{Gridchyn2013, IwataY2016}, a minimizer of $g$ (an optimal labeling) is partly recovered from a minimizer of the relaxation $\tilde{g}$.
Gridchyn--Kolmogorov~\cite{Gridchyn2013} showed that a minimizer of a Potts $k$-submodular function can be obtained by $\Order{\log k}$ calls of a max-flow algorithm performed on a network of $\Order{n}$ vertices. 
We show that the PIP is also obtained in the same time complexity (Theorem~\ref{thm:Potts}).
In showing this result, we reveal an intriguing structure of the PIP for a Potts $k$-submodular function (Theorem~\ref{thm:PIP_Potts}), and utilize results~\cite{Hirai2010, Ibaraki1998} from undirected multiflow theory.

We also discuss enumeration aspects for minimizers.
Maximal minimizers, which are minimizers with a maximum number of nonzero components, are of particular interest from the view of partial optimal labeling.
For a Potts $k$-submodular function, we show that the problem of enumerating all maximal minimizers reduces to the problem of enumerating all ideals of a single poset (Theorem~\ref{thm:enumeration_potts}).
This enables us to use an existing fast enumeration algorithm, and leads to a practical algorithm enumerating all maximal partial optimal labeling in actual computer vision problems. 
We present experimental results for real instances of stereo matching problems. 

\paragraph{Organization.}
The rest of this paper is organized as follows.
In Section~\ref{sec:preliminaries}, we give preliminaries including a Birkhoff-type representation theorem between PIPs and median semilattices.
In Section~\ref{sec:structure}, we prove the above-mentioned structural results.
In Section~\ref{sec:algorithms}, we prove algorithmic results.
Finally, in Section~\ref{sec:application}, we describe applications and present experimental results.

\section{Preliminaries}
\label{sec:preliminaries}

For a nonnegative integer $n$, we denote $\setprn{1, 2, \ldots, n}$ by $\intset{n}$ (with $\intset{0} \defeq \varnothing$).
For a subset $X$ of an ordered set, let $\min X$ denote the minimum element in $X$ (if it exists).
Let $\setR$ be the set of real numbers and $\setRext \defeq \setR \cup \setprn{+\infty}$.
For a function $f$ from a set $D$ to $\setRext$, a \emph{minimizer} of $f$ is an element $x \in D$ that satisfies $\app{f}{x} \leq \app{f}{y}$ for all $y \in D$.
The set of minimizers of $f$ is simply called the \emph{minimizer set} of $f$.
We assume that posets are always finite, and assume the standard notions of lattice theory, such as join $\join$ and meet $\meet$.

\subsection{$k$-submodular function}
Let $k$ be a positive integer.
Let $S_k$ denote $\setprn{0, 1, 2, \ldots, k}$.
The partial order $\preceq$ on $S_k$ is defined by $a \preceq b$ if and only if $a \in \setprn{0, b}$ for each $a, b \in S_k$.
Consider the $n$-product ${S_k}^n$ of $S_k$, where the partial order on ${S_k}^n$ is defined as the direct product of $\preceq$ and is also denoted by $\preceq$.
In this way, ${S_k}^n$ and its subsets are regarded as posets.
For $x = \paren{x_1, x_2, \ldots, x_n} \in {S_k}^n$, the \emph{support} of $x$ is the set of indices $i \in \intset{n}$ with nonzero $x_i$, and is denoted by $\supp x$:
\begin{align*}
  \supp x \defeq \setprnsep{i \in \intset{n}}{x_i \neq 0}.
\end{align*}

A \emph{$k$-submodular function}~\cite{Huber2012} is a function $\funcdoms{f}{{S_k}^n}{\setRext}$ satisfying the following inequalities
\begin{align} \label{neq:k-submodularity}
  \app{f}{x} + \app{f}{y} \geq \app{f}{x \sqmeet y} + \app{f}{x \sqjoin y}
\end{align}
for all $x, y \in {S_k}^n$.
Here the binary operation $\sqmeet$ on ${S_k}^n$ is given by
\begin{align} \label{def:sqmeet}
  \paren{x \sqmeet y}_i \defeq \begin{cases}
    \min \setprn{x_i, y_i} & \text{($x_i$ and $y_i$ are comparable with respect to $\preceq$)}, \\
    0                      & \text{($x_i$ and $y_i$ are incomparable with respect to $\preceq$)},
  \end{cases}
\end{align}
for every $x, y \in {S_k}^n$ and $i \in \intset{n}$.
The operation $\sqjoin$ in \eqref{neq:k-submodularity} is defined by changing $\min$ to $\max$ in \eqref{def:sqmeet}.

Besides its recent introduction, a $k$-submodular function seems to be recognized when Bouchet~\cite{Bouchet1997} introduced \emph{multimatroids}.
Indeed, a $k$-submodular function is a direct generalization of the rank function of a multimatroid, and was suggested by Fujishige~\cite{Fujishige1995} in 1995 as a \emph{multisubmodular function}.

It is not known whether $k$-submodular functions for $k \geq 3$ can be minimized in polynomial time on the standard oracle model.
However, some special classes of $k$-submodular functions are efficiently minimizable.
For example, Kolmogorov--Thapper--\v{Z}ivn\'{y}~\cite{Kolmogorov2015} showed that a sum of low-arity $k$-submodular functions can be minimized in polynomial time, where the \emph{arity} of a function is the number of variables.
A nonnegative combination of binary \emph{basic $k$-submodular functions}, introduced by Iwata--Wahlstr\"{o}m--Yoshida~\cite{IwataY2016}, can be minimized by computing a minimum $\paren{s,t}$-cut on a directed network; see Section~\ref{sec:network}.

A nonempty subset of ${S_k}^n$ is said to be \emph{$\paren{\sqmeet, \sqjoin}$-closed} if it is closed under the operations $\sqmeet$ and $\sqjoin$.
From~\eqref{neq:k-submodularity}, the following obviously holds.

\begin{lemma}
  The minimizer set of a $k$-submodular function is $\paren{\sqmeet, \sqjoin}$-closed.
\end{lemma}

\subsection{Median semilattice and PIP}
A key tool for providing a compact representation for $\paren{\sqmeet, \sqjoin}$-closed sets is a correspondence between median semilattices and PIPs, which was established by Barth\'{e}lemy--Constantin~\cite{Barthelemy1993}.
A recent paper~\cite{Chepoi2012} also contains an exposition of this correspondence.

A \emph{median semilattice}~\cite{Sholander1954} is a meet-semilattice $L = \paren{L, \leq}$ satisfying the following conditions:
\begin{namedenum}{MS}
  \item Every principal ideal is a distributive lattice.
  \item For all $x, y, z \in L$, if $x \join y \wc y \join z$ and $z \join x$ exist, then $x \join y \join z$ exists in $L$.
\end{namedenum}
Note that every distributive lattice is a median semilattice.
An element of $L$ is said to be \emph{join-irreducible} if it is not minimum and is not represented as a join of other elements.
Let $\irreducibles{L}$ denote the set of join-irreducible elements of $L$.

Next we introduce a \emph{poset with inconsistent pairs (PIP)}.
A PIP~\cite{Ardila2011, Barthelemy1993, Nielsen1981} is a poset $P = \paren{P, \leq}$ endowed with an additional symmetric relation $\inconsistent$ satisfying the following conditions:
\begin{namedenum}{IC}
  \item For all $p, q \in P$ with $p \inconsistent q$, there is no $r \in P$ with $p \leq r$ and $q \leq r$.
  \item For all $p, q, p', q' \in P$, if $p' \leq p, q' \leq q$ and $p' \inconsistent q'$, then $p \inconsistent q$.
\end{namedenum}
A PIP is also denoted by a triple $\paren{P, \leq, \inconsistent}$.
The relation $\inconsistent$ is called an \emph{inconsistency relation}.
Each unordered pair $\setprn{p, q}$ of $P$ is called \emph{inconsistent} if $p \inconsistent q$.
Note that every inconsistent pair of $P$ is incomparable.
An inconsistent pair $\setprn{p, q}$ of $P$ is said to be \emph{minimally inconsistent} if $p' \leq p$, $q' \leq q$ and $p' \inconsistent q'$ imply $p = p'$ and $q = q'$ for all $p', q' \in P$.
If $\setprn{p, q}$ is minimally inconsistent, the $p \inconsistent q$ is particularly denoted by $p \mininconsistent q$.
We can easily check the following properties of the minimal inconsistency relation:
\begin{namedenum}{MIC}
  \item For all $p, q \in P$ with $p \mininconsistent q$, there is no $r \in P$ with $p \leq r$ and $q \leq r$.
  \item For all $p, q, p', q' \in P$ with $p' \leq p$ and $q' \leq q$, if $p' \mininconsistent q'$ and $p \mininconsistent q$, then $p' = p$ and $q' = q$.
\end{namedenum}
Actually, PIPs can also be defined as a triple $\paren{P, \leq, \mininconsistent}$, where $\mininconsistent$ is a binary symmetric relation on a poset $P = \paren{P, \leq}$ satisfying the conditions (MIC1) and (MIC2).
In this definition, the inconsistency relation $\inconsistent$ on $P$ is obtained by
\begin{align*}
  \text{$p \inconsistent q$ if and only if there exist $p', q' \in P$ with $p' \leq p$, $q' \leq q$ and $p' \mininconsistent q'$}
\end{align*}
for every $p, q \in P$.
Since both definitions of PIP are equivalent, we will use a convenient one.

For a PIP $P$, an ideal of $P$ is said to be \emph{consistent} if it contains no (minimally) inconsistent pair.
Let $\consideals{P}$ denote the family of consistent ideals of $P$.
Regard $\consideals{P}$ as a poset with respect to the inclusion order $\subseteq$.
\begin{figure}[tb]
  \begin{minipage}[t][2.7cm][t]{0.33\linewidth}
    \centering
    \begin{tikzpicture}[every node/.style={circle, draw, scale=0.8}, x=15pt, y=20pt]
      \node (a) at (0, 0) {};
      \node (b) at (2, 0) {};
      \node (c) at (1, 1) {};
      \node[draw=none, scale=1.0] (p) at (0.4, 1) {$p$};
      \node (d) at (4, 0) {};
      \node (e) at (4, 1) {};
      \node[draw=none, scale=1.0] (q) at (4.6, 1) {$q$};
      \node (f) at (0, 2) {};
      \node (g) at (2.5, 2) {};
      \node[draw=none, scale=1.0] (r) at (2.5, 2.4) {$r$};
      \foreach \u / \v in {a/c, b/c, d/e, c/f, c/g, e/g}
        \draw[->, >=latex'] (\v) -- (\u);
      \draw[line width=0.3mm, dash pattern=on .03mm off .8mm, line cap=round] (c) to (e);
      \draw[line width=0.3mm, dash pattern=on .03mm off .8mm, line cap=round] (e) to (f);
    \end{tikzpicture}
    \subcaption{violating (IC1)}
  \end{minipage}%
  \begin{minipage}[t][2.7cm][t]{0.33\linewidth}
    \centering
    \begin{tikzpicture}[every node/.style={circle, draw, scale=0.8}, x=15pt, y=20pt]
      \node (a) at (0, 0) {};
      \node (b) at (2, 0) {};
      \node (c) at (1, 1) {};
      \node (d) at (4, 0) {};
      \node[draw=none, scale=1.0] (p) at (4.6, 0) {$p$};
      \node (e) at (4, 1) {};
      \node[draw=none, scale=1.0] (qr) at (5.1, 1) {$q=r$};
      \node (f) at (0, 2) {};
      \foreach \u / \v in {a/c, b/c, d/e, c/f}
        \draw[->, >=latex'] (\v) -- (\u);
      \draw[line width=0.3mm, dash pattern=on .03mm off .8mm, line cap=round] (c) to (e);
      \draw[line width=0.3mm, dash pattern=on .03mm off .8mm, line cap=round] (d) to [bend right=40] (e);
      \draw[line width=0.3mm, dash pattern=on .03mm off .8mm, line cap=round] (e) to (f);
    \end{tikzpicture}
    \subcaption{violating (IC1)}
  \end{minipage}%
  \begin{minipage}[t][2.7cm][t]{0.33\linewidth}
    \centering
    \begin{tikzpicture}[every node/.style={circle, draw, scale=0.8}, x=15pt, y=20pt]
      \node (a) at (0, 0) {};
      \node (b) at (2, 0) {};
      \node (c) at (1, 1) {};
      \node (d) at (4, 0) {};
      \node (e) at (4, 1) {};
      \node (f) at (0, 2) {};
      \foreach \u / \v in {a/c, b/c, d/e, c/f}
        \draw[->, >=latex'] (\v) -- (\u);
      \draw[densely dashed] (c) to (e);
      \draw[line width=0.3mm, dash pattern=on .03mm off .8mm, line cap=round] (e) to (f);
    \end{tikzpicture}
    \subcaption{PIP}
  \end{minipage}
  \begin{minipage}[t][2.5cm][t]{0.5\linewidth}
    \centering
    \begin{tikzpicture}[every node/.style={circle, draw, scale=0.8}, x=15pt, y=20pt]
      \node (a) at (0, 0) {};
      \node (b) at (2, 0) {};
      \node[draw=none, scale=1.0] (p') at (1.4, 0) {$p'$};
      \node (c) at (1, 1) {};
      \node[draw=none, scale=1.0] (p') at (0.4, 1) {$p$};
      \node (d) at (4, 0) {};
      \node[draw=none, scale=1.0] (p') at (4.65, 0) {$q'$};
      \node (e) at (4, 1) {};
      \node[draw=none, scale=1.0] (p') at (4.65, 1) {$q$};
      \node (f) at (0, 2) {};
      \foreach \u / \v in {a/c, b/c, d/e, c/f}
        \draw[->, >=latex'] (\v) -- (\u);
      \draw[line width=0.3mm, dash pattern=on .03mm off .8mm, line cap=round] (b) to (d);
    \end{tikzpicture}
    \subcaption{violating (IC2)}
  \end{minipage}%
  \begin{minipage}[t][2.5cm][t]{0.5\linewidth}
    \centering
    \begin{tikzpicture}[every node/.style={circle, draw, scale=0.8}, x=15pt, y=20pt]
      \node (a) at (0, 0) {};
      \node (b) at (2, 0) {};
      \node (c) at (1, 1) {};
      \node (d) at (4, 0) {};
      \node (e) at (4, 1) {};
      \node (f) at (0, 2) {};
      \foreach \u / \v in {a/c, b/c, d/e, c/f}
        \draw[->, >=latex'] (\v) -- (\u);
      \foreach \u / \v in {b/e, d/c, d/f, c/e, e/f}
        \draw[line width=0.3mm, dash pattern=on .03mm off .8mm, line cap=round] (\v) -- (\u);
      \draw[densely dashed] (b) to (d);
    \end{tikzpicture}
    \subcaption{PIP}
  \end{minipage}
  \caption{
    Examples of PIPs and non-PIP structures.
    Solid arrows indicate the orders between elements (drawn from higher elements to lowers).
    Dotted lines and dashed lines indicate the inconsistency relations.
    In (a), (b) and (d), labeled elements indicate where the violations of (IC1) and (IC2) are.
    In (c) and (e), the minimal inconsistency relations are drawn by dashed lines.
  }
  \label{fig:pip}
\end{figure}
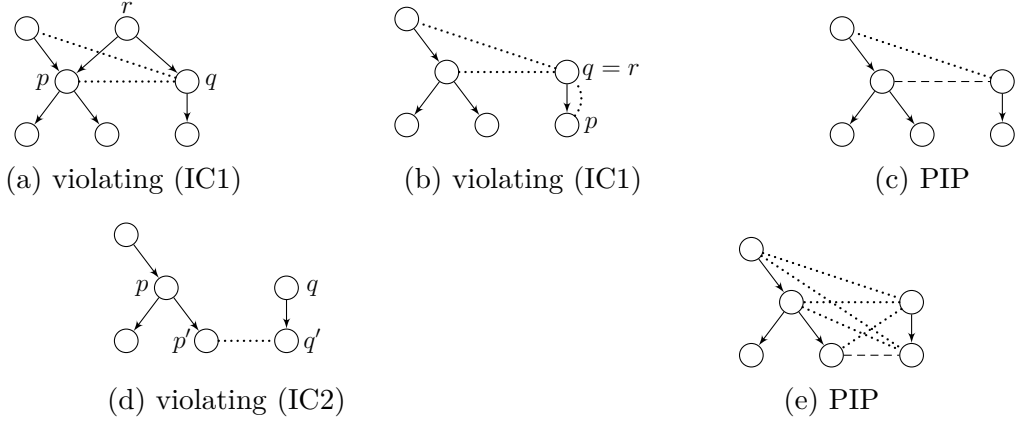
Figure~\ref{fig:pip} shows examples of PIPs and non-PIP structures.

The following theorem establishes a one-to-one correspondence between median semilattices and PIPs.

\begin{theorem}[{\cite[Theorem 2.16]{Barthelemy1993}}] \label{thm:median_semilattice_and_pip}
  \begin{enumerate}
    \item Let $L = \paren{L, \leq}$ be a median semilattice and $\inconsistent$ a symmetric binary relation on $L$ defined by
          \begin{align*}
            \text{$x \inconsistent y$ if and only if $x \join y$ does not exist in $L$}
          \end{align*}
          for every $x, y \in L$.
          Then $\paren{\irreducibles{L}, \leq, \inconsistent}$ forms a PIP with inconsistency relation $\inconsistent$.
          The consistent ideal family $\consideals{\irreducibles{L}}$ is isomorphic to $L$, and an isomorphism is given by $I \mapsto \bigjoin_{x \in I} x$ for $I \neq \varnothing$ and $\varnothing \mapsto \min L$.
    
    \item Let $P$ be a PIP.
          The consistent ideal family $\consideals{P}$ forms a median semilattice.
          The PIP $\paren{\irreducibles{\consideals{P}}, \subseteq, \inconsistent}$ obtained in the same way as (1) is isomorphic to $P$.
  \end{enumerate}
\end{theorem}
The latter part of Theorem~\ref{thm:median_semilattice_and_pip}~(2) is implicit in~\cite{Barthelemy1993}, and follows from Theorem~\ref{thm:median_semilattice_and_pip}~(1) and the fact that for PIPs $P$ and $P'$, if $\consideals{P}$ and $\consideals{P'}$ are isomorphic, then $P$ and $P'$ are also isomorphic~\cite[p.57]{Barthelemy1993}.

\begin{remark} \label{lem:pip_and_cnf}
  A PIP is an alternative expression of a satisfiable Boolean 2-CNF, where consistent ideals correspond to true assignments.
  Indeed, for a PIP $\paren{P, \leq, \inconsistent}$ with $P = \intset{n}$, consider the following 2-CNF of Boolean variables $x_1, x_2, \dots, x_n \in \setprn{0, 1}$:
  \begin{align*}
    \paren{\bigland_\condition{i,j \in P}{i < j} x_i \lor \bar{x}_j}
    \land
    \paren{\bigland_\condition{i,j \in P}{i \inconsistent j} \bar{x}_i \lor \bar{x}_j}.
  \end{align*}
  Then an assignment $\paren{x_1, x_2, \ldots, x_n} \in \setprn{0, 1}^n$ is true if and only if the set of elements $i \in P$ with $x_i = 1$ is a consistent ideal.
  The reverse construction of a PIP from a 2-CNF satisfiable at $\paren{0, 0, \ldots, 0}$ is also easily verified.
\end{remark}

\section{$\paren{\sqmeet, \sqjoin}$-closed set and elementary PIP}
\label{sec:structure}

The starting point for a compact representation for $\paren{\sqmeet, \sqjoin}$-closed sets is the following.
\begin{lemma} \label{lem:closed_set_is_median_semilattice}
  Every $\paren{\sqmeet, \sqjoin}$-closed set is a median semilattice.
\end{lemma}
\begin{proof}
  Let $M \subseteq {S_k}^n$ be a $\paren{\sqmeet, \sqjoin}$-closed set.
  Then $M$ is a semilattice since ${S_k}^n$ is a semilattice with minimum element $\bm{0} \defeq \paren{0, 0, \ldots, 0} \in {S_k}^n$, and the operator $\sqmeet$ coincides with $\meet$ on ${S_k}^n$.
  We show that $M$ satisfies the conditions (MS1) and (MS2).
  
  (MS1).
  Let $I$ be the principal ideal of $x \in M$.
  For all $y \in I$ and $i \in \intset{n}$, $y_i$ is equal to either 0 or $x_i$.
  Therefore, for all $y, z \in I$, the join $y \join z$ exists and it holds $y \join z = y \sqjoin z \in I$.
  Next let $\funcdoms{\phi}{I}{\Power{\intset{n}}}$ be an injection defined by $\app{\phi}{y} \defeq \supp y$ for every $y \in I$.
  One can easily see that $\app{\phi}{y \meet z} = \app{\phi}{y} \cap \app{\phi}{z}$ and $\app{\phi}{y \join z} = \app{\phi}{y} \cup \app{\phi}{z}$ for every $y, z \in I$.
  In other words, $\phi$ is an isomorphism from $\paren{I, \preceq}$ to $\paren{\app{\phi}{I}, \subseteq}$.
  Since any nonempty subset of $\Power{\intset{n}}$ closed under $\cap$ and $\cup$ is a distributive lattice ordered by inclusion, $I$ is also distributive.
  
  (MS2).
  Let $x, y, z \in M$ be such that the join of any two of them exists in $M$.
  Since $x_i, y_i$ and $z_i$ are comparable for any $i \in \intset{n}$, the join $x \join y \join z$ exists in ${S_k}^n$, and coincides with $x \sqjoin y \sqjoin z$.
  Finally since $M$ is closed under $\sqjoin$, the join $x \join y \join z$ belongs to $M$.
\end{proof}

Let $\inconsistent$ be a symmetric binary relation on ${S_k}^n$ defined by
\begin{align*}
  \text{$x \inconsistent y$ if and only if $x \join y$ does not exist in ${S_k}^n$}
\end{align*}
for every $x, y \in {S_k}^n$.
Note that for every $x, y \in M$, if $x \not\inconsistent y$ then $x \join y$ is equal to $x \sqjoin y$.
From Theorem~\ref{thm:median_semilattice_and_pip}~(1) and Lemma~\ref{lem:closed_set_is_median_semilattice}, we obtain the following.
\begin{theorem} \label{thm:closed_set_and_pip}
  Let $M$ be a $\paren{\sqmeet, \sqjoin}$-closed set.
  Then $\paren{\irreducibles{M}, \preceq, \inconsistent}$ forms a PIP with inconsistency relation $\inconsistent$.
  The consistent ideal family $\consideals{\irreducibles{M}}$ is isomorphic to $M$, and the isomorphism is given by $I \mapsto \bigjoin_{x \in I} x$ for $I \neq \varnothing$ and $\varnothing \mapsto \min M$.
\end{theorem}

\begin{figure}[t]
  \begin{minipage}[t][3cm][t]{0.5\linewidth}
    \centering
    \begin{tikzpicture}[every node/.style={rectangle, draw, scale=0.8}, x=40pt, y=20pt]
      \node [] (000) at ( 0,   0) {00002};
      \node [double]  (100) at (-1,   1) {13002};
      \node [double]  (001) at ( 1,   1) {00012};
      \node [double]  (110) at (-1.5, 2) {13102};
      \node [] (101) at (-0.5, 2) {13012};
      \node [double]  (120) at ( 0.5, 2) {13202};
      \node [double]  (031) at ( 1.5, 2) {00312};
      \node [] (111) at (-1.5, 3) {13112};
      \node [] (121) at (-0.5, 3) {13212};
      \node [] (131) at ( 0.5, 3) {13312};
      \node [double]  (231) at ( 1.5, 3) {22312};
      \foreach \u / \v in {
        000/100, 000/001, 100/110, 100/101, 100/120, 001/101, 001/031,
        110/111, 101/111, 101/121, 101/131, 120/121, 031/131, 031/231}
        \draw[->, >=latex'] (\v) -- (\u);
    \end{tikzpicture}
    \subcaption{a $(\sqmeet, \sqjoin)$-closed set on ${S_3}^5$}
  \end{minipage}%
  \begin{minipage}[t][3cm][t]{0.5\linewidth}
    \centering
    \begin{tikzpicture}[every node/.style={rectangle, draw, double, scale=0.8}, x=35pt, y=25pt]
      \node (100) at (-1,   0.5) {13002};
      \node (110) at (-1.5, 1.5) {13102};
      \node (120) at (-0.5, 1.5) {13202};
      \node (001) at ( 1,   0)   {00012};
      \node (031) at ( 1,   1)   {00312};
      \node (231) at ( 1,   2)   {22312};
      \foreach \u / \v in {100/110, 100/120, 001/031, 031/231}
        \draw[->, >=latex'] (\v) -- (\u);
      \foreach \u / \v in {110/120, 120/031, 100/231}
        \draw[densely dashed] (\v) -- (\u);
      \draw[densely dashed] (110) to [bend right=17] (031);
    \end{tikzpicture}
    \subcaption{the PIP corresponding to (a)}
  \end{minipage}
  \caption{
    Example of a $(\sqmeet, \sqjoin)$-closed set and the corresponding PIP.
    In (a), elements surrounded by double-lined frames are join-irreducible.
    In (b), non-minimal inconsistency relations are not drawn.
  }
  \label{fig:example_of_closed_set_and_pip}
\end{figure}
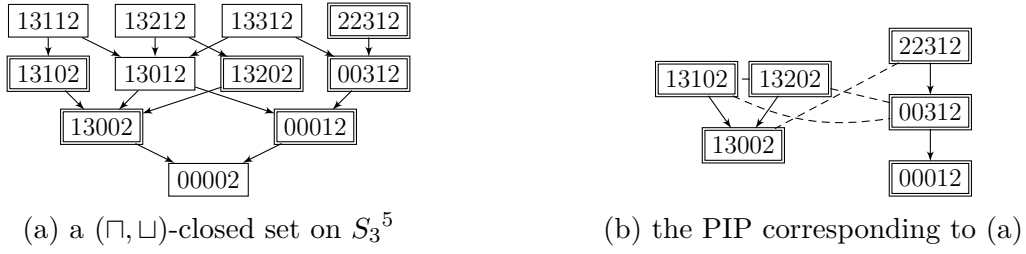

Figure~\ref{fig:example_of_closed_set_and_pip} shows an example of a $(\sqmeet, \sqjoin)$-closed set and the corresponding PIP.

From Theorem~\ref{thm:closed_set_and_pip}, it will turn out that the set $\irreducibles{M}$ of join-irreducible elements of every $\paren{\sqmeet, \sqjoin}$-closed set $M$ does not lose any information about the structure of $M$.
That is, non-minimum elements in $M$ can be obtained as the join of one or more join-irreducible elements of $M$ (notice that we cannot obtain the minimum element of $M$ in this way).
Therefore we call $\irreducibles{M}$ a \emph{PIP-representation} of $M$.
Furthermore, the following proposition, which will be proved in Section~\ref{sec:proof_of_theorem_1}, says that this representation is actually compact.
\begin{proposition} \label{prop:number_of_irreducibles}
  Let $M$ be a $\paren{\sqmeet, \sqjoin}$-closed set on ${S_k}^n$.
  The number of join-irreducible elements of $M$ is at most $kn$.
\end{proposition}

Theorem~\ref{thm:closed_set_and_pip} states that any $\paren{\sqmeet, \sqjoin}$-closed set can be represented by a PIP.
However, not all PIPs correspond to some $\paren{\sqmeet, \sqjoin}$-closed sets.
A natural question then arises: \textit{What class of PIPs represents $\paren{\sqmeet, \sqjoin}$-closed sets?}
The main result (Theorem~\ref{thm:closed_set_and_elementary_pip}) of this section answers this question.

\begin{definition} \label{def:elementary}
  A PIP $P = \paren{P, \leq}$ is called \emph{elementary} if it satisfies the following conditions:
  \begin{namedenum}{EP}
    \item[(EP0)] $P$ is the disjoint union of $P_1, P_2, \ldots, P_n$ such that every pair $\setprn{x, y} \subseteq P$ of distinct elements is minimally inconsistent if and only if $\setprn{x, y} \subseteq P_i$ for some $i \in \intset{n}$.
    \item[(EP1)] For any distinct $i, j \in \intset{n}$, if $\absprn{P_i} \geq 2$ and $P_j = \setprn{y}$, there is no element $x \in P_i$ with $x < y$.
    \item[(EP2)] For any distinct $i, j \in \intset{n}$, if $\absprn{P_i} \geq 2$ and $\absprn{P_j} \geq 2$, either of the following two holds:
    \begin{namedenum}{EP}
      \item[(EP2-1)] Every pair of $x \in P_i$ and $y \in P_j$ is not comparable.
      \item[(EP2-2)] There exist $x^\circ \in P_i$ and $y^\circ \in P_j$ such that $x^\circ < y$ and $y^\circ < x$ for all $x \in P_i \setminus \setprn{x^\circ}$ and $y \in P_j \setminus \setprn{y^\circ}$.
    \end{namedenum}
  \end{namedenum}
\end{definition}

\begin{figure}[t]
  \begin{minipage}[t][1.9cm][t]{0.25\linewidth}
    \centering
    \begin{tikzpicture}[every node/.style={circle, draw, scale=0.8}, x=20pt, y=20pt]
      \node (01) at (0, 0) {};
      \node (02) at (1, 0) {};
      \node (03) at (2, 0) {};
      \node[fill] (10) at (0, 1) {};
      \draw[->, >=latex'] (10) -- (01);
      \draw[densely dashed] (01) -- (02);
      \draw[densely dashed] (02) -- (03);
      \draw[densely dashed] (01.south east) to [bend right=25] (03.south west);
    \end{tikzpicture}
    \subcaption{violating (EP1)}
  \end{minipage}%
  \begin{minipage}[t][1.9cm][t]{0.25\linewidth}
    \centering
    \begin{tikzpicture}[every node/.style={circle, draw, scale=0.8}, x=20pt, y=20pt]
      \node (01) at (0, 1) {};
      \node (02) at (1, 1) {};
      \node (03) at (2, 1) {};
      \node[fill] (10) at (0, 0) {};
      \draw[->, >=latex'] (01) -- (10);
      \draw[densely dashed] (01) -- (02);
      \draw[densely dashed] (02) -- (03);
      \draw[densely dashed] (01.north east) to [bend left=25] (03.north west);
    \end{tikzpicture}
    \subcaption{elementary}
  \end{minipage}%
  \begin{minipage}[t][1.9cm][t]{0.25\linewidth}
    \centering
    \begin{tikzpicture}[every node/.style={circle, draw, scale=0.8}, x=20pt, y=20pt]
      \node (10) at (0, 0) {};
      \node (20) at (1, 0) {};
      \node (30) at (2, 0) {};
      \node[fill] (01) at (0, 1) {};
      \node[fill] (02) at (1, 1) {};
      \node[fill] (03) at (2, 1) {};
      \draw[->, >=latex'] (01) -- (10);
      \draw[->, >=latex'] (02) -- (20);
      \draw[->, >=latex'] (03) -- (30);
      \draw[densely dashed] (01) -- (02);
      \draw[densely dashed] (02) -- (03);
      \draw[densely dashed] (01.north east) to [bend left=25] (03.north west);
      \draw[densely dashed] (10) -- (20);
      \draw[densely dashed] (20) -- (30);
      \draw[densely dashed] (10.south east) to [bend right=25] (30.south west);
    \end{tikzpicture}
    \subcaption{violating (EP2)}
  \end{minipage}%
  \begin{minipage}[t][1.9cm][t]{0.25\linewidth}
    \centering
    \begin{tikzpicture}[every node/.style={circle, draw, scale=0.8}, x=20pt, y=20pt]
      \node (10) at ( 0,   0) {};
      \node[draw=none, scale=1.0] (x) at (-0.5, 0.05) {$x^\circ$};
      \node[fill] (02) at (-0.5, 1) {};
      \node[fill] (03) at ( 0.5, 1) {};
      \node[fill] (01) at ( 2,   0) {};
      \node[draw=none, scale=1.0] (x) at (2.5, 0.05) {$y^\circ$};
      \node (20) at ( 1.5, 1) {};
      \node (30) at ( 2.5, 1) {};
      \draw[->, >=latex'] (02) -- (10);
      \draw[->, >=latex'] (03) -- (10);
      \draw[->, >=latex'] (20) -- (01);
      \draw[->, >=latex'] (30) -- (01);
      \draw[densely dashed] (01) -- (03);
      \draw[densely dashed] (02) -- (03);
      \draw[densely dashed] (02) to [bend right=10] (01);
      \draw[densely dashed] (10) -- (20);
      \draw[densely dashed] (20) -- (30);
      \draw[densely dashed] (10) to [bend right=10] (30);
    \end{tikzpicture}
    \subcaption{elementary}
  \end{minipage}
  \caption{
    Examples of elementary PIPs and non-elementary PIPs.
    In all diagrams, the drawn PIPs satisfy the condition (EP0) with $n = 2$.
    Each element is filled or not filled according to the corresponding part $P_i$.
    Non-minimal inconsistency relations are not drawn in each diagram.
  }
  \label{fig:example_of_elementary_pip}
\end{figure}
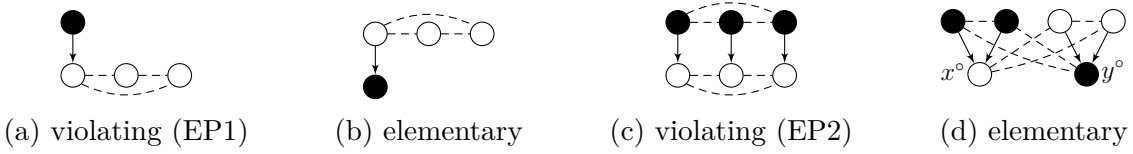

Figure~\ref{fig:example_of_elementary_pip} shows examples of elementary PIPs and non-elementary PIPs.

\begin{theorem} \label{thm:closed_set_and_elementary_pip}
  \begin{enumerate}
    \item For every $\paren{\sqmeet, \sqjoin}$-closed set $M$, the PIP $\paren{\irreducibles{M}, \preceq, \inconsistent}$ is elementary.
    \item For every elementary PIP $P$, there is a $\paren{\sqmeet, \sqjoin}$-closed set $M$ isomorphic to $\consideals{P}$.
  \end{enumerate}
\end{theorem}

An elementary PIP corresponds to a $\paren{\sqmeet, \sqjoin}$-closed set on the product of the most ``elementary'' median semilattice $S_k$, whereas general PIP can represent an arbitrary median semilattice (by Theorem~\ref{thm:median_semilattice_and_pip}).
This is why we use the term ``elementary.''

\begin{remark}
  Consider an elementary PIP $P$ with the property that each $P_i$ has the cardinality at most 2.
  Such a PIP arises from $\paren{\sqmeet, \sqjoin}$-closed sets on ${S_2}^n$.
  If we assign a sign $+,-$ to each element so that two nodes in $P_i$ have a different sign, then the PIP is equivalently transformed into a \emph{signed poset}~\cite{Reiner1993}, which is a certain ``acyclic and transitive'' bidirected graph and is used by Ando--Fujishige~\cite{Ando1994} for representing $\paren{\sqmeet, \sqjoin}$-closed sets in ${S_2}^n$.
  Then \emph{ideals} in the signed poset correspond to consistent ideals in the original PIP.
  In the transformation, elements in the signed poset are nonempty members in $P_1, P_2, \ldots, P_n$.
  Bidirected edges are given according to an appropriate rule; one can guess the rule from the example in Figure~\ref{fig:example_of_signed_poset}. (In this figure, we omit redundant edges derived from the transitive closure.)
  In this way, one can see that the PIP-representation for $\paren{\sqmeet, \sqjoin}$-closed sets on ${S_2}^n$ is equivalent to the one by Ando--Fujishige~\cite{Ando1994}.

  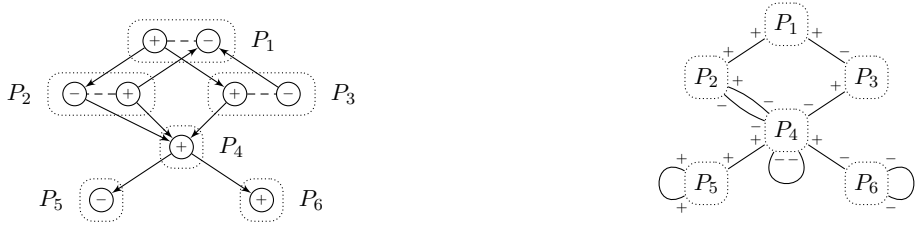
\begin{figure}[t]
    \begin{minipage}[t][3.4cm][t]{0.5\linewidth}
      \centering
      \begin{tikzpicture}[every node/.style={circle, draw, scale=0.8}, x=10pt, y=10pt]
        \node[label={[label distance=-15pt]\tiny $+$}] (1p) at (-1,  0) {};
        \node[label={[label distance=-15pt]\tiny $-$}] (1n) at ( 1,  0) {};
        \node[draw, rectangle, rounded corners=5pt, minimum width=50pt, minimum height=20pt, densely dotted, label=right:{\small $P_1$}] at (0, 0) {};
        \draw[densely dashed] (1p) -- (1n);
        \node[label={[label distance=-15pt]\tiny $+$}] (2p) at (-2, -2) {};
        \node[label={[label distance=-15pt]\tiny $-$}] (2n) at (-4, -2) {};
        \node[draw, rectangle, rounded corners=5pt, minimum width=50pt, minimum height=20pt, densely dotted, label=left:{\small $P_2$}] at (-3, -2) {};
        \draw[densely dashed] (2p) -- (2n);
        \node[label={[label distance=-15pt]\tiny $+$}] (3p) at ( 2, -2) {};
        \node[label={[label distance=-15pt]\tiny $-$}] (3n) at ( 4, -2) {};
        \node[draw, rectangle, rounded corners=5pt, minimum width=50pt, minimum height=20pt, densely dotted, label=right:{\small $P_3$}] at (3, -2) {};
        \node[label={[label distance=-15pt]\tiny $+$}] (4)  at ( 0, -4) {};
        \node[draw, rectangle, rounded corners=5pt, minimum width=20pt, minimum height=20pt, densely dotted, label=right:{\small $P_4$}] at (0, -4) {};
        \draw[densely dashed] (3p) -- (3n);
        \node[label={[label distance=-15pt]\tiny $-$}] (5)  at (-3, -6) {};
        \node[draw, rectangle, rounded corners=5pt, minimum width=20pt, minimum height=20pt, densely dotted, label=left:{\small $P_5$}] at (-3, -6) {};
        \node[label={[label distance=-15pt]\tiny $+$}] (6)  at ( 3, -6) {};
        \node[draw, rectangle, rounded corners=5pt, minimum width=20pt, minimum height=20pt, densely dotted, label=right:{\small $P_6$}] at (3, -6) {};
        \draw[->, >=latex'] (1p) -- (2n);
        \draw[->, >=latex'] (1p) -- (3p);
        \draw[->, >=latex'] (2p) -- (1n);
        \draw[->, >=latex'] (3n) -- (1n);
        \draw[->, >=latex'] (2p) -- (4);
        \draw[->, >=latex'] (2n) -- (4);
        \draw[->, >=latex'] (3p) -- (4);
        \draw[->, >=latex'] (4) -- (5);
        \draw[->, >=latex'] (4) -- (6);
      \end{tikzpicture}
      \subcaption{an elementary PIP with assigned signs}
    \end{minipage}%
    \begin{minipage}[t][3.4cm][t]{0.5\linewidth}
      \centering
      \begin{tikzpicture}[every node/.style={circle, draw, scale=0.8}, x=10pt, y=10pt]
        \node[draw, rectangle, rounded corners=5pt, minimum width=20pt, minimum height=20pt, densely dotted] (1) at (0, 0) {\small $P_1$};
        \node[draw, rectangle, rounded corners=5pt, minimum width=20pt, minimum height=20pt, densely dotted] (2) at (-3, -2) {\small $P_2$};
        \node[draw, rectangle, rounded corners=5pt, minimum width=20pt, minimum height=20pt, densely dotted] (3) at ( 3, -2) {\small $P_3$};
        \node[draw, rectangle, rounded corners=5pt, minimum width=20pt, minimum height=20pt, densely dotted] (4) at ( 0, -4) {\small $P_4$};
        \node[draw, rectangle, rounded corners=5pt, minimum width=20pt, minimum height=20pt, densely dotted] (5) at (-3, -6) {\small $P_5$};
        \node[draw, rectangle, rounded corners=5pt, minimum width=20pt, minimum height=20pt, densely dotted] (6) at ( 3, -6) {\small $P_6$};
        \draw[font=\tiny] (1) to node [draw=none, very near start, xshift=-2pt, yshift=4pt] {$+$} node [draw=none, very near end, xshift=-2pt, yshift=4pt] {$+$} (2);
        \draw[font=\tiny] (1) to node [draw=none, very near start, xshift=2pt, yshift=4pt] {$+$} node [draw=none, very near end, xshift=2pt, yshift=4pt] {$-$} (3);
        \draw[shorten >= -2pt, font=\tiny, bend left=10] (2) to node [draw=none, very near start, xshift=2pt, yshift=4pt] {$+$} node [draw=none, very near end, xshift=4pt, yshift=2pt] {$-$} (4);
        \draw[shorten <= -2pt, font=\tiny, bend right=10] (2) to node [draw=none, very near start, xshift=-5pt, yshift=-2pt] {$-$} node [draw=none, very near end, xshift=-1.5pt, yshift=-4pt] {$-$} (4);
        \draw[font=\tiny] (3) to node [draw=none, very near start, xshift=-2pt, yshift=4pt] {$+$} node [draw=none, very near end, xshift=-2pt, yshift=4pt] {$-$} (4);
        \draw[font=\tiny] (4) to node [draw=none, very near start, xshift=-2pt, yshift=4pt] {$+$} node [draw=none, very near end, xshift=-2pt, yshift=4pt] {$+$} (5);
        \draw[font=\tiny] (4) to node [draw=none, very near start, xshift=2pt, yshift=4pt] {$+$} node [draw=none, very near end, xshift=2pt, yshift=4pt] {$-$} (6);
        \draw[font=\tiny] (4) to [loop, in=240, out=300, looseness=5] node [draw=none, very near start, xshift=-5pt, yshift=3pt] {$-$} node [draw=none, very near end, xshift=5pt, yshift=3pt] {$-$} (4);
        \draw[font=\tiny] (5) to [loop, in=150, out=210, looseness=4] node [draw=none, very near start, xshift=4pt, yshift=-4pt] {$+$} node [draw=none, very near end, xshift=4pt, yshift=4pt] {$+$} (5);
        \draw[font=\tiny] (6) to [loop, in=-30, out=30, looseness=4] node [draw=none, very near start, xshift=-4pt, yshift=4pt] {$-$} node [draw=none, very near end, xshift=-4pt, yshift=-4pt] {$-$} (6);
      \end{tikzpicture}
      \subcaption{the singed poset corresponding to (a)}
    \end{minipage}
    \caption{Example of an elementary PIP on ${S_2}^n$ and the corresponding signed poset.}
    \label{fig:example_of_signed_poset}
  \end{figure}

\end{remark}

The following corollary of Theorem~\ref{thm:median_semilattice_and_pip}~(2) and Theorem~\ref{thm:closed_set_and_elementary_pip}~(1) will be used in Section~\ref{sec:algorithms}.

\begin{corollary} \label{cor:elementary_pip}
  Let $P$ be a PIP.
  If $\consideals{P}$ is isomorphic to some $\paren{\sqmeet, \sqjoin}$-closed set, then $P$ is elementary.
\end{corollary}

The remaining part of this section is devoted to proving Theorem~\ref{thm:closed_set_and_elementary_pip}.
To get a motivation behind the properties of elementary PIPs which we prove below, readers may choose to read Algorithm~\ref{alg:get_irreducibles_minimizers} in Section~\ref{sec:minimizing_oracle} first.

\subsection{Proof of Theorem~\ref{thm:closed_set_and_elementary_pip}~(1)}
\label{sec:proof_of_theorem_1}

The proof of Theorem~\ref{thm:closed_set_and_elementary_pip}~(1) is outlined as follows:
\begin{enumerate}
  \item[1.] First we define the \emph{differential} of a join-irreducible element $x$ as the difference between $x$ and the unique lower cover $x'$ of $x$.
  \item[2.] Next we introduce a \emph{normalized} $\paren{\sqmeet, \sqjoin}$-closed set, which is a $\paren{\sqmeet, \sqjoin}$-closed set such that every differential has exactly one nonzero component.
            We show that every $\paren{\sqmeet, \sqjoin}$-closed set is isomorphic to some normalized $\paren{\sqmeet, \sqjoin}$-closed set.
            This set gives us a natural partition of join-irreducible elements.
  \item[3.] Finally we construct an elementary PIP from the partition.
\end{enumerate}

A $\paren{\sqmeet, \sqjoin}$-closed set $M$ is said to be \emph{simple} if $\min M = \bm{0}$.
Any $\paren{\sqmeet, \sqjoin}$-closed set can be converted to a simple $\paren{\sqmeet, \sqjoin}$-closed set without any structural change.

\begin{definition} \label{def:differential}
  Let $M \subseteq {S_k}^n$ be a simple $\paren{\sqmeet, \sqjoin}$-closed set.
  For $x, y \in M$, we say that $y$ is a \emph{lower cover} of $x$, or \emph{$x$ covers $y$}, if $y \prec x$ and there is no $z \in M$ such that $y \prec z \prec x$.
  For a join-irreducible element $x \in \irreducibles{M}$, there uniquely exists $y \in M$ covered by $x$.
  The \emph{differential} $\bar{x} \in {S_k}^n$ of $x$ is defined by $\bar{x}_i \defeq x_i$ if $x_i \succ y_i = 0$ and $\bar{x}_i \defeq 0$ if $x_i = y_i$, for each $i \in \intset{n}$.
\end{definition}

The uniqueness of a lower cover of a join-irreducible element $x \in \irreducibles{M}$ can be seen from the fact that if $x$ has two or more lower covers, then $x$ is obtained as the join of these lower covers.
\begin{figure}[t]
  \begin{minipage}[t][3.7cm][t]{0.5\linewidth}
    \centering
    \begin{tikzpicture}[every node/.style={rectangle, draw, scale=0.8}, x=40pt, y=20pt]
      \node [] (000) at ( 0,   0) {0000};
      \node [double]  (100) at (-1,   1) {1300};
      \node [double]  (001) at ( 1,   1) {0001};
      \node [double]  (110) at (-1.5, 2) {1310};
      \node [] (101) at (-0.5, 2) {1301};
      \node [double]  (120) at ( 0.5, 2) {1320};
      \node [double]  (031) at ( 1.5, 2) {0031};
      \node [] (111) at (-1.5, 3) {1311};
      \node [] (121) at (-0.5, 3) {1321};
      \node [] (131) at ( 0.5, 3) {1331};
      \node [double]  (231) at ( 1.5, 3) {2231};
      \foreach \u / \v in {
        000/100, 000/001, 100/110, 100/101, 100/120, 001/101, 001/031,
        110/111, 101/111, 101/121, 101/131, 120/121, 031/131, 031/231}
        \draw[->, >=latex'] (\v) -- (\u);
    \end{tikzpicture}
    \subcaption{a simple $(\sqmeet, \sqjoin)$-closed set on ${S_3}^4$}
  \end{minipage}%
  \begin{minipage}[t][3.7cm][t]{0.5\linewidth}
    \centering
    \begin{tikzpicture}[every node/.style={rectangle, draw, double, scale=0.7, text width=width("1300")}, x=35pt, y=30pt]
      \node (100) at (-1,   0.5) {1300\\\textit{1300}};
      \node (110) at (-1.5, 1.5) {1310\\\textit{0010}};
      \node (120) at (-0.5, 1.5) {1320\\\textit{0020}};
      \node (001) at ( 1,   0)   {0001\\\textit{0001}};
      \node (031) at ( 1,   1)   {0031\\\textit{0030}};
      \node (231) at ( 1,   2)   {2231\\\textit{2200}};
      \foreach \u / \v in {100/110, 100/120, 001/031, 031/231}
        \draw[->, >=latex'] (\v) -- (\u);
      \foreach \u / \v in {110/120, 120/031, 100/231}
        \draw[densely dashed] (\v) -- (\u);
      \draw[densely dashed] (110) to [bend right=20] (031);
    \end{tikzpicture}
    \subcaption{the PIP corresponding to (a)}
  \end{minipage}
  \begin{minipage}[t][3.3cm][t]{0.5\linewidth}
    \centering
    \begin{tikzpicture}[every node/.style={rectangle, draw, scale=0.8}, x=40pt, y=20pt]
      \node [] (000) at ( 0,   0) {000};
      \node [double]  (100) at (-1,   1) {100};
      \node [double]  (001) at ( 1,   1) {001};
      \node [double]  (110) at (-1.5, 2) {110};
      \node [] (101) at (-0.5, 2) {101};
      \node [double]  (120) at ( 0.5, 2) {120};
      \node [double]  (031) at ( 1.5, 2) {031};
      \node [] (111) at (-1.5, 3) {111};
      \node [] (121) at (-0.5, 3) {121};
      \node [] (131) at ( 0.5, 3) {131};
      \node [double]  (231) at ( 1.5, 3) {231};
      \foreach \u / \v in {
        000/100, 000/001, 100/110, 100/101, 100/120, 001/101, 001/031,
        110/111, 101/111, 101/121, 101/131, 120/121, 031/131, 031/231}
        \draw[->, >=latex'] (\v) -- (\u);
    \end{tikzpicture}
    \subcaption{a normalized $(\sqmeet, \sqjoin)$-closed set on ${S_3}^3$}
  \end{minipage}%
  \begin{minipage}[t][3.3cm][t]{0.5\linewidth}
    \centering
    \begin{tikzpicture}[every node/.style={rectangle, draw, scale=0.7, double, text width=width("100")}, x=35pt, y=30pt]
      \node (100) at (-1,   0.5) {100\\\textit{100}};
      \node (110) at (-1.5, 1.5) {110\\\textit{010}};
      \node (120) at (-0.5, 1.5) {120\\\textit{020}};
      \node (001) at ( 1,   0)   {001\\\textit{001}};
      \node (031) at ( 1,   1)   {031\\\textit{030}};
      \node (231) at ( 1,   2)   {231\\\textit{200}};
      \foreach \u / \v in {100/110, 100/120, 001/031, 031/231}
        \draw[->, >=latex'] (\v) -- (\u);
      \foreach \u / \v in {110/120, 120/031, 100/231}
        \draw[densely dashed] (\v) -- (\u);
      \draw[densely dashed] (110) to [bend right=20] (031);
    \end{tikzpicture}
    \subcaption{the PIP corresponding to (c)}
  \end{minipage}
  \caption{
    Examples of a simple $(\sqmeet, \sqjoin)$-closed set, a normalized set and the corresponding PIPs.
    In (b) and (d), the differential of each join-irreducible element is written in italics.
  }
  \label{fig:example_of_differential_and_normalizing}
\end{figure}
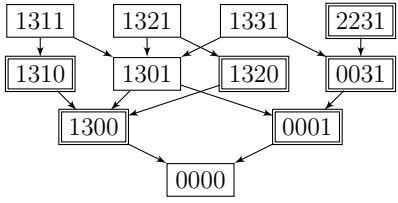
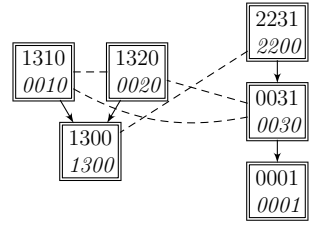
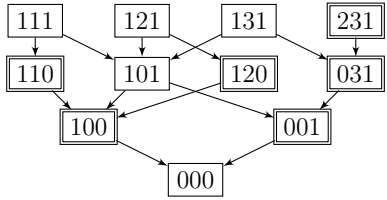
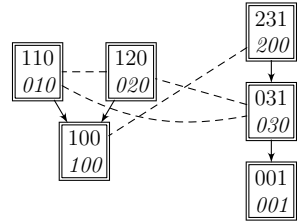
Figure~\ref{fig:example_of_differential_and_normalizing}~(a) and (b) show examples of a simple $\paren{\sqmeet, \sqjoin}$-closed set and the corresponding PIP, respectively.

We show some properties about differentials.
In what follows, we denote the subset $\setprnsep{x \in M}{x_i = \alpha}$ by $\partfix{M}{i}{\alpha}$ for $i \in \intset{n}$ and $\alpha \in \intset{k}$.
Note that $\partfix{M}{i}{\alpha}$ also forms a $\paren{\sqmeet, \sqjoin}$-closed set if $\partfix{M}{i}{\alpha} \neq \varnothing$.

\begin{lemma} \label{lem:differential}
  Let $M \subseteq {S_k}^n$ be a simple $\paren{\sqmeet, \sqjoin}$-closed set. 
  The following hold:
  \begin{enumerate}
    \item For every $i \in \intset{n}$ and $\alpha \in \intset{k}$ with $\partfix{M}{i}{\alpha} \neq \varnothing$, $x \defeq \min \partfix{M}{i}{\alpha}$ is join-irreducible in $M$ and $\bar{x}_i = \alpha$ holds.
    \item For every $x \in \irreducibles{M}$ and $i \in \intset{n}$ with $\alpha \defeq \bar{x}_i \neq 0$, it holds $x = \min \partfix{M}{i}{\alpha}$.
    \item For every $i \in \intset{n}$ and $\alpha \in \intset{k}$, there is at most one join-irreducible element $x \in \irreducibles{M}$ such that $\bar{x}_i = \alpha$.
    \item For every $x \in \irreducibles{M}$, the differential $\bar{x}$ of $x$ has at least one nonzero component.
    \item The map $x \mapsto \bar{x}$ is an injection from $\irreducibles{M}$ to ${S_k}^n$.
  \end{enumerate}
\end{lemma}
\begin{proof}
  (1).
  Let $i \in \intset{n} \wc \alpha \in \intset{k}$ and $x \defeq \min \partfix{M}{i}{\alpha}$.
  Suppose to the contrary that $x \notin \irreducibles{M}$.
  Then there exist $y, z \in M$ such that $x \succ y, z$ and $x = y \join z$.
  Since $\alpha = x_i = y_i \join z_i$, either $y_i$ or $z_i$ is equal to $\alpha$.
  This contradicts the assumption that $x = \min \partfix{M}{i}{\alpha}$ and $x \succ y, z$.
  Hence $x$ is join-irreducible.
  Moreover, from Definition~\ref{def:differential}, it holds $\bar{x}_i = \alpha$.
    
  (2).
  Let $x \in \irreducibles{M}$ and $i \in \intset{n}$ such that $\alpha \defeq \bar{x}_i \neq 0$.
  Then $x \in \partfix{M}{i}{\alpha}$.
  Let $y \in M$ be the lower cover of $x$ and let $z \in \partfix{M}{i}{\alpha}$ be the minimum element of $\partfix{M}{i}{\alpha}$.
  Suppose that $x \neq z$.
  Then it holds $z \preceq y \prec x$ since $z \prec x$.
  Hence we obtain $z_i = y_i = x_i = \alpha$, which claims that $\bar{x}_i = 0$ by Definition~\ref{def:differential}.
  This contradicts the assumption.
  Thus $x = z = \min \partfix{M}{i}{\alpha}$ holds.
  
  (3).
  Suppose that $M$ has a join-irreducible element $x \in \irreducibles{M}$ such that $\bar{x}_i = \alpha$.
  From (2), it holds $x = \min \partfix{M}{i}{\alpha}$.
  This lemma follows from the uniqueness of the minimum element of $\partfix{M}{i}{\alpha}$.
  
  (4).
  Assume that $M$ has a join-irreducible element $x$ such that $\bar{x} = \bm{0}$.
  Let $y \in M$ be the lower cover of $x$.
  Then $x_i = y_i$ must hold for each $i \in \intset{n}$, which contradicts that $y \prec x$.
  
  (5).
  Let $x, y \in \irreducibles{M}$ such that $\bar{x} = \bar{y}$.
  Since $\bar{x} = \bar{y} \neq \bm{0}$ from (4), there exists $i \in \intset{n}$ such that $\bar{x}_i = \bar{y}_i \neq 0$.
  Then $x = y$ follows from (3).
\end{proof}

Now Proposition~\ref{prop:number_of_irreducibles} is a consequence of Lemma~\ref{lem:differential}.

\begin{proof}[Proof of Proposition~\ref{prop:number_of_irreducibles}]
  It suffices to consider the case where $M$ is simple.
  From Lemmas~\ref{lem:differential}~(3) and (4), it holds $\absprn{\setprnsep{\bar{x}}{x \in \irreducibles{M}}} \leq kn$.
  Furthermore, since the map $x \to \bar{x}$ is injective, we have $\absprn{\irreducibles{M}} = \absprn{\setprnsep{\bar{x}}{x \in \irreducibles{M}}}$.
  Hence $\absprn{\irreducibles{M}}$ is at most $kn$.
\end{proof}

A simple $\paren{\sqmeet, \sqjoin}$-closed set $M$ is said to be \emph{normalized} if it satisfies $\absprn{\supp \bar{x}} = 1$ for all $x \in \irreducibles{M}$.
Examples of a normalized $\paren{\sqmeet, \sqjoin}$-closed set and the corresponding PIP are shown in Figure~\ref{fig:example_of_differential_and_normalizing}~(c) and (d), respectively.

\begin{lemma} \label{lem:normalized_is_general}
  For any $\paren{\sqmeet, \sqjoin}$-closed set $M$, there exists a normalized $\paren{\sqmeet, \sqjoin}$-closed set that is isomorphic to $M$ with respect to the relations $\preceq$ and $\smile$.
\end{lemma}
\begin{proof}
  We can suppose that $M$ is simple.
  We first show:
  \begin{enumerate}
    \item[(1)] For $x, y \in \irreducibles{M}$, it holds that $\supp \bar{x} = \supp \bar{y}$ or $\supp \bar{x} \cap \supp \bar{y} = \varnothing$.
  \end{enumerate}
  Suppose to the contrary that there exist $x,y \in \irreducibles{M}$ such that $\bar{x}_i \neq 0 \neq \bar{y}_i$ and $\bar{x}_j \neq 0 = \bar{y}_j$ for distinct $i, j \in \intset{n}$.
  Then it holds $\bar{x}_i \neq \bar{y}_i$ from Lemma~\ref{lem:differential}~(3).
  Hence we have $x \not\preceq x \sqjoin y$ since $\paren{x \sqjoin y}_i = 0$.
  However, both $x$ and $x \sqjoin y$ belong to $\partfix{M}{j}{x_j}$ since $\paren{x \sqjoin y}_j = x_j \neq 0$, thus it holds $x \preceq x \sqjoin y$ from Lemma~\ref{lem:differential}~(2).
  This is a contradiction.
  
  By (1), we can define an equivalence relation $\sim$ over the index set
  \begin{align*}
    \setprnsep{i \in \intset{n}}{\text{there exists $x \in \irreducibles{M}$ such that $\bar{x}_i \neq 0$}}
  \end{align*}
  as follows:
  \begin{align*}
    \text{$i \sim j$ if and only if there exists $x \in \irreducibles{M}$ such that $\bar{x}_i \neq 0$ and $\bar{x}_j \neq 0$}.
  \end{align*}
  Then each equivalence class can be ``contracted'' into a single index without any structural change of $M$ as follows.
  Let $\setprn{I_1, I_2, \ldots, I_{\tilde{n}}}$ be the set of equivalence classes.
  For $j \in \intset{\tilde n}$, let $\setprn{x^{j,1}, x^{j,2}, \ldots, x^{j,k_{j}}} \subseteq \irreducibles{M}$ be the set of join-irreducible elements having the differentials of support $I_j$.
  Then, by Lemma~\ref{lem:differential},
  \begin{enumerate}
    \item[(2)] For every $x \in M$ and $j \in \intset{\tilde{n}}$, either $x_i = 0$ for all $i \in I_j$ or there uniquely exists $\alpha \in \intset{k_j}$ such that $x_i = \paren{x^{j, \alpha}}_i$ for all $i \in I_j$.
  \end{enumerate}
  Let $\tilde{k} \defeq \max_{j \in [\tilde n]} k_j$.
  Define $\funcdoms{\phi}{M}{{S_{\tilde{k}}}^{\tilde{n}}}$ by $\app{\phi}{x}_j = 0$ if $x_i = 0$ for $i \in I_j$, and $\app{\phi}{x}_j = \alpha$ if $x_i = \paren{x^{j,\alpha}}_i$ for $i \in I_j$.
  It is easily verified (from Lemma~\ref{lem:differential}) that the map $\phi$ is injective and preserves $\preceq$ and $\inconsistent$.
  An irreducible element of $\app{\phi}{M}$ is the image of an irreducible element of $M$, and, by construction, has the differential of a singleton support.  
  Thus $\app{\phi}{M}$ is a normalized $\paren{\sqmeet, \sqjoin}$-closed set.
\end{proof}

Now we are ready to prove Theorem~\ref{thm:closed_set_and_elementary_pip}~(1).

\begin{proof}[Proof of Theorem~\ref{thm:closed_set_and_elementary_pip}~(1)]
  By Lemma~\ref{lem:normalized_is_general}, it suffices to consider the case where $M$ is normalized.
  For every $i \in \intset{n}$, let $J_i \defeq \setprnsep{x \in \irreducibles{M}}{\supp \bar{x} = \setprn{i}}$.
  From the definition of normalized $\paren{\sqmeet, \sqjoin}$-closed sets, $\setprn{J_1, J_2, \ldots, J_n}$ forms a partition of $\irreducibles{M}$ (note that $J_i$ may be empty).
  We show that the PIP $\paren{\irreducibles{M}, \preceq, \smile}$ satisfies the axiom of elementary PIPs with $P_i = J_i$ for every $i \in \intset{n}$.
  
  (EP0, ``only if'' part).
  Let $\setprn{x, y} \subseteq \irreducibles{M}$ be a minimally inconsistent pair.
  Then there exists $i \in \intset{n}$ such that $0 \neq x_i \neq y_i \neq 0$.
  We show $0 \neq \bar{x}_i \neq \bar{y}_i \neq 0$.
  Suppose that $\bar{x}_i = 0$.
  From Definition~\ref{def:differential}, there exists $x' \in \irreducibles{M}$ such that $x' \prec x$ and $x'_i = x_i$.
  Now we have $x' \prec x$ and $x' \smile y$, which contradict the assumption that $x$ and $y$ are minimally inconsistent.
  Therefore $\bar{x}_i \neq 0$ holds, and we can show $\bar{y}_i \neq 0$ in the same way.
  Thus $\setprn{x, y} \subseteq J_i$ holds.
  
  (EP1) is an immediate consequence of the following property:
  \begin{itemize}
    \item[($*$)] Let $i, j \in \intset{n}$ be distinct.
      If there exist $x' \in J_i$ and $y \in J_j$ such that $x' \prec y$, then for all $x \in J_i \setminus \setprn{x'}$, there exists $y' \in J_j \setminus \setprn{y}$ such that $y'
\prec x$; in particular $|J_j| \geq 2$ if $|J_i| \geq 2$.
  \end{itemize}
  
  We show ($*$).
  Let $x' \in J_i$ and $y \in J_j$ such that $x' \prec y$.
  Now it holds $x'_i = y_i \neq 0$ and $y_j \neq 0$.
  Let $x \in J_i \setminus \setprn{x'}$.
  We have $y_i \neq \paren{x \sqjoin y}_i = 0$ since $0 \neq x_i \neq x'_i = y_i \neq 0$.
  Thus $y \preceq x \sqjoin y$ does not hold.
  We show that $0 \neq x_j \neq y_j$.
  If not, $\paren{x \sqjoin y}_j$ is equal to $y_j$, hence $x \sqjoin y$ belongs to $\partfix{M}{j}{y_j}$.
  Therefore it holds $y \preceq x \sqjoin y$ since $y$ is the minimum element of $\partfix{M}{j}{y_j}$ from Lemma~\ref{lem:differential}~(2).
  We have a contradiction here.
  Let $y' \defeq \min \partfix{M}{j}{x_j}$.
  This $y'$ belongs to $J_j$ from Lemma~\ref{lem:differential}~(1), and it holds $y \neq y' \prec x$.
  
  (EP2).
  Let $i, j \in \intset{n}$ be distinct indices such that $\absprn{J_i} \geq 2$ and $\absprn{J_j} \geq 2$.
  We can assume that (EP2-1) does not hold, i.e., there exist $x' \in J_i$ and $y \in J_j$ such that $x' \prec y$.
  Consider $x, z \in J_i \setminus \setprn{x'}$.
  By ($*$), there exist $y', y'' \in J_j \setminus \setprn{y}$ such that $y' \prec x$ and $y'' \prec z$.
  We show $y' = y''$. Suppose not.
  Since $y' \prec x$ and $y'' \in J_j \setminus \setprn{y'}$, we can take $x'' \in J_i \setminus \setprn{x}$ such that $x'' \prec y''$ by ($*$) with changing the role of $i$ and $j$.
  Now we have $x'' \prec y'' \prec z$, which contradicts $x'' \smile z$.
  Therefore $y'$ and $y''$ are same elements.
  Consequently, the required element $y^\circ$ in (EP2-2) is given by $y'$.
  By changing the role of $i$ and $j$, we see that $x^\circ$ is given by $x'$.
  
  (EP0, ``if'' part).
  Let $x, y \in J_i$ be distinct with $i \in \intset{n}$.
  Now since $x \smile y$, there exists a minimally inconsistent pair $\setprn{x', y'} \subseteq \irreducibles{M}$ such that $x' \preceq x$ and $y' \preceq y$.
  From the ``only if'' part of (EP0), $x'$ and $y'$ belong to $J_j$ for some $j \in \intset{n}$.
  If $i \neq j$, then we have $x, y \in J_i$, $x',y' \in J_j$, $x' \prec x$ and $y' \prec y$, which contradict (EP2).
  Hence $i = j$ and it must hold $\setprn{x', y'} = \setprn{x, y}$.
\end{proof}

\subsection{Proof of Theorem~\ref{thm:closed_set_and_elementary_pip}~(2)}

Let $\paren{P, \leq, \inconsistent}$ be an elementary PIP with partition $\setprn{P_1, P_2, \ldots, P_n}$ of condition (EP0).
For every $i \in \intset{n}$, let $P_i = \setprn{e^{i,1}, e^{i,2}, \ldots, e^{i,k_i}}$, where $k_i \defeq \absprn{P_i}$.
Let $k \defeq \max_{i \in \intset{n}} k_i$.
For a consistent ideal $I$, let $\app{x}{I} = \paren{\app{x_1}{I}, \app{x_2}{I}, \ldots, \app{x_n}{I}} \in {S_k}^n$ be defined by
\begin{align*}
  \app{x_i}{I} \defeq \begin{cases}
    \alpha & \text{($I \cap P_i = \setprn{e^{i,\alpha}}$),} \\
    0      & \text{($I \cap P_i = \varnothing$)}
  \end{cases}
  \quad
  \paren{i \in \intset{n}}.
\end{align*}
Now $\app{x}{I}$ is well-defined since every consistent ideal $I$ of $P$ has at most one element in each $P_i$ by (EP0).
Let $M \defeq \setprnsep{\app{x}{I}}{I \in \consideals{P}}$.
Then $\consideals{P}$ and $M$ are clearly isomorphic.
Therefore the rest of the proof of Theorem~\ref{thm:closed_set_and_elementary_pip}~(2) is to show that $M$ forms a $\paren{\sqmeet, \sqjoin}$-closed set.

We define binary operations $\sqmeet$ and $\sqjoin$ on $\consideals{P}$ as $I \sqmeet J \defeq I \cap J$ and
\begin{align*}
  I \sqjoin J \defeq \bigcup_{i=1}^n \setprnsep{p \in P_i}{\paren{I \cup J} \cap P_i = \setprn{p}}
\end{align*}
for every $I, J \in \consideals{P}$.
Theorem~\ref{thm:closed_set_and_elementary_pip}~(2) follows immediately from:

\begin{lemma} \label{lem:consistent_ideal_family_of_P_is_closed}
  For every $I, J \in \consideals{P}$, it hold $I \sqmeet J \in \consideals{P}, I \sqjoin J \in \consideals{P}$, $\app{x}{I \sqmeet J} = \app{x}{I} \sqmeet \app{x}{J}$ and $\app{x}{I \sqjoin J} = \app{x}{I} \sqjoin \app{x}{J}$.
\end{lemma}
\begin{proof}
  Let $I, J \in \consideals{P}$.
  Since the consistent ideal family of a PIP is closed under the intersection, $I \sqmeet J$ also forms a consistent ideal of $P$.
  In addition, we can easily check that $\app{x}{I \sqmeet J} = \app{x}{I} \sqmeet \app{x}{J}$ holds.
  
  Next we consider $I \sqjoin J$.
  We show that $I \sqjoin J$ is a consistent ideal.
  Suppose that $I \sqjoin J$ is not an ideal of $P$.
  There exist $q \in I \sqjoin J$ and $p \in P \setminus \paren{I \sqjoin J}$ such that $p < q$.
  Without loss of generality, we assume $q \in I$.
  Now $I$ also contains $p$ since $I$ is an ideal.
  Let $i, j \in \intset{n}$ such that $p \in P_i$ and $q \in P_j$.
  We can take $r \in \paren{J \cap P_i} \setminus \setprn{p}$ since $p \notin I \sqjoin J$.
  Thus $\absprn{P_i}$ is greater than 1, and $\absprn{P_j}$ is also greater than 1 since $p < q$ contradicts the condition (EP1) if $\absprn{P_j} = 1$.
  From (EP2-2), there exists $s \in P_j \setminus \setprn{q}$ such that $s < r$.
  It holds $s \in J$ since $r \in J$.
  Now we have $q \neq s$, $q \in I$, $s \in J$ and $q, s \in P_j$.
  This contradicts $q \in I \sqjoin J$.
  Therefore $I \sqjoin J$ is an ideal.
  Finally suppose that $I \sqjoin J$ includes an inconsistent pair $\setprn{p, q}$.
  Since $I \sqjoin J$ is an ideal, it also includes the minimally inconsistent pair $\setprn{p', q'}$ with $p' \leq p$ and $q' \leq q$.
  From (EP0), $p'$ and $q'$ belong to the same part $P_i$ of the partition.
  This contradicts the fact that $\absprn{\paren{I \sqjoin J} \cap P_i} \leq 1$, and thus $I \sqjoin J$ is a consistent ideal.
  $\app{x}{I \sqjoin J} = \app{x}{I} \sqjoin \app{x}{J}$ follows from the definitions of $\sqjoin$ on ${S_k}^n$ and on $\consideals{P}$.
\end{proof}

\section{Algorithms}
\label{sec:algorithms}

In this section, we study algorithmic aspects of constructing PIP-representations for the minimizer sets of $k$-submodular functions.
Let $\minimizers{f}$ denote the minimizer set of a function $f$.
Let $\maxflow{n}{m}$ denote the time complexity of an algorithm of a maximum flow (and a minimum cut) in a network of $n$ vertices and $m$ edges.
We assume a standard max-flow algorithm, such as preflow-push algorithm, and hence assume that $\maxflow{n}{m}$ is not less than $\Order{nm}$; notice that the current fastest one is
an $\Order{nm}$ algorithm by Orlin~\cite{Orlin2013}.

\subsection{By a minimizing oracle}
\label{sec:minimizing_oracle}
We can obtain the PIP-representation for the minimizer set of a $k$-submodular function $\funcdoms{f}{{S_k}^n}{\setRext}$ by using a minimizing oracle $k$-SFM, which returns a minimizer of $f$ and its restrictions.
Let $\min f$ be the minimum value of $f$.
For $i \in \intset{n}$ and $a \in S_k$, we define a new $k$-submodular function $\funcdoms{\fixarg{f}{i}{a}}{{S_k}^n}{\setRext}$ from $f$ by
\begin{align*}
  \app{\fixarg{f}{i}{a}}{x_1, \ldots, x_i, \ldots, x_n} \defeq f(x_1, \ldots, \vectorpos{a}{i}, \ldots, x_n)
  \quad
  \paren{x \in {S_k}^n}.
\end{align*}
Namely, $\fixarg{f}{i}{a}$ is a function obtained by fixing the $i$-th variable of $f$ to $a$.

Before describing the main part of our algorithm, we present a subroutine \textproc{GetMinimumMinimizer} in Algorithm~\ref{alg:get_minimum_minimizer}.
This subroutine returns the minimum minimizer of a $k$-submodular function.
The validity of this subroutine can be checked by the fact that $\min \fixarg{f}{i}{0}$ is equal to $\min f$ if $\paren{\min \minimizers{f}}_i = 0$ and otherwise it holds $\min \fixarg{f}{i}{0} > \min f$.
This subroutine calls $k$-SFM at most $n+1$ times.

\begin{algorithm}[tp]
  \caption{Obtain the minimum minimizer of a $k$-submodular function}
  \label{alg:get_minimum_minimizer}
  \begin{algorithmic}[1]
    \Input  A $k$-submodular function $\funcdoms{f}{{S_k}^n}{\setRext}$
    \Output The minimum minimizer $\min \minimizers{f}$ of $f$
    \Function{GetMinimumMinimizer}{$f$}
      \State{$x \gets \Call{$k$-SFM}{f}$}
      \For{$i \in \supp x$}
        \If{$\min \fixarg{f}{i}{0} = \min f$}
          \State{$x_i \gets 0$}
        \EndIf
      \EndFor
      \State{\Return $x$}
    \EndFunction
  \end{algorithmic}
\end{algorithm}

Algorithm~\ref{alg:get_irreducibles_minimizers} shows a procedure to collect all join-irreducible minimizers of a $k$-submodular function.
Let $x$ be the minimum minimizer of $f$.
The function $\funcdoms{\tilde{f}}{{S_k}^n}{\setRext}$ in Algorithm~\ref{alg:get_irreducibles_minimizers} is defined as $\app{\tilde{f}}{y} \defeq \app{f}{\paren{y \sqjoin x} \sqjoin x}$ for every $y \in {S_k}^n$.
Since $\paren{\paren{y \sqjoin x} \sqjoin x}_i$ is equal to $y_i$ if $x_i = 0$ and to $x_i$ if $x_i \neq 0$, we can regard $\tilde{f}$ as a $k$-submodular function obtained by fixing each $i$-th variable of $f$ to $x_i$ if $x_i \neq 0$.
Note that the minimum values of $f$ and $\tilde{f}$ are the same.
The correctness of this algorithm is based on Lemma~\ref{lem:differential}~(1) and (2).
Namely, the set of join-irreducible minimizers of $f$ coincides with the set
\begin{align} \label{set:min}
  \setprnsep{\min \minimizers{\fixarg{\tilde{f}}{i}{\alpha}}}{i \in \intset{n} \setminus \supp{x} \wc \alpha \in \intset{k} \wc \min \fixarg{\tilde{f}}{i}{\alpha} = \min f}.
\end{align}
The algorithm collects each join-irreducible minimizer according to \eqref{set:min} by calling \textproc{GetMinimumMinimizer} at most $nk+1$ times.
Consequently, if a minimizing oracle is available, the minimizer set can also be obtained in polynomial time.

\begin{theorem} \label{thm:alg_oracle}
  The PIP-representation for the minimizer set of a $k$-submodular function $\funcdoms{f}{{S_k}^n}{\setRext}$ is obtained by $\Order{kn^2}$ calls of $k$-SFM.
\end{theorem}

\begin{algorithm}[tp]
  \caption{Collect all join-irreducible minimizers of a $k$-submodular function}
  \label{alg:get_irreducibles_minimizers}
  \begin{algorithmic}[1]
    \Input  A $k$-submodular function $\funcdoms{f}{{S_k}^n}{\setRext}$
    \Output The set $\irreducibles{\minimizers{f}}$ of all join-irreducible minimizers of $f$
    \Function{GetJoinIrreducibleMinimizers}{$f$}
      \State{$x \defeq \Call{GetMinimumMinimizer}{f}$}
      \State{$\tilde{f} \defeq $ the function obtained by fixing the $i$-th variable of $f$ to $x_i$ for all $i \in \supp x$}
      \State{$J \gets \varnothing$}
      \For{$i \in \intset{n} \setminus \supp x$}
        \ForTo{$\alpha \gets 1$}{$k$}
          \If{$\min \fixarg{\tilde{f}}{i}{\alpha} = \min f$}
            \State{$J \gets J \cup \setprn{\Call{GetMinimumMinimizer}{\fixarg{\tilde{f}}{i}{\alpha}}}$}
          \EndIf
        \EndFor
      \EndFor
      \State{\Return $J$}
    \EndFunction
  \end{algorithmic}
\end{algorithm}

\subsection{Network-representable $k$-submodular functions}
\label{sec:network}

Iwata--Wahlstr\"{o}m--Yoshida~\cite{IwataY2016} introduced \emph{basic $k$-submodular functions}, which form a special class of $k$-submodular functions.
They showed a reduction of the minimization problem of a nonnegative combination of binary basic $k$-submodular functions to the minimum cut problem on a directed network.
We describe their method and present an algorithm to obtain the PIP-representation for the minimizer set.

Let $n$ and $k$ be positive integers.
We consider a directed network $N = \paren{V, A, c}$ with vertex set $V$, edge set $A$ and nonnegative edge capacity $c$.
Suppose that $V$ consists of source $s$, sink $t$ and other vertices $v_i^\alpha$, where $i \in \intset{n}$ and $\alpha \in \intset{k}$.
Let $U_i \defeq \setprn{v_i^1, v_i^2, \ldots, v_i^k}$ for $i \in \intset{n}$.
An $\paren{s,t}$-cut of $N$ is a subset $X$ of $V$ such that $s \in X$ and $t \notin X$.
We call an $\paren{s,t}$-cut $X$ \emph{legal} if $\absprn{X \cap U_i} \leq 1$ for every $i \in \intset{n}$.
There is a natural bijection $\psi$ from ${S_k}^n$ to the set of legal $\paren{s,t}$-cuts of $N$ defined by
\begin{align*}
  \app{\psi}{x} \defeq \setprn{s} \cup \setprnsep{v_i^{x_i}}{i \in \supp x}
  \quad
  \paren{x \in {S_k}^n}.
\end{align*}
See Figure~\ref{fig:example_of_legal_cut}.
\begin{figure}[t]
  \centering
  \begin{tikzpicture}[every node/.style={scale=0.8}, x=8pt, y=8pt]
    \tikzstyle{vertex}=[circle, draw, scale=0.6];
    \tikzstyle{group}=[circle, draw, rectangle, rounded corners=10pt, minimum width=70pt, minimum height=30pt, densely dotted, label=above:{#1}];
    \node[vertex] (11) at (0, 0) {};
    \node[vertex] (12) at (2, 0) {};
    \node[vertex] (13) at (4, 0) {};
    \node[group=$U_1$] at (2, 0) {};
    \node[vertex] (21) at (8, 0) {};
    \node[vertex] (22) at (10, 0) {};
    \node[vertex] (23) at (12, 0) {};
    \node[group=$U_2$] at (10, 0) {};
    \node[vertex] (31) at (16, 0) {};
    \node[vertex] (32) at (18, 0) {};
    \node[vertex] (33) at (20, 0) {};
    \node[group=$U_3$] at (18, 0) {};
    \node[vertex] (41) at (24, 0) {};
    \node[vertex] (42) at (26, 0) {};
    \node[vertex] (43) at (28, 0) {};
    \node[group=$U_4$] at (26, 0) {};
    \node[vertex] (51) at (32, 0) {};
    \node[vertex] (52) at (34, 0) {};
    \node[vertex] (53) at (36, 0) {};
    \node[group=$U_5$] at (34, 0) {};
    \node[vertex, fill, above=35pt of 32, label=above right:{$t$}] (t) {};
    \node[vertex, fill, below=25pt of 32, label=above right:{$s$}] (s) {};
    \draw[thick] (-1, -1)
      .. controls +(0, 0)   and +(0, 0)   .. (-1, 0)
      .. controls +(0, 1.5) and +(0, 1.5) .. (1, 0)
      .. controls +(0, 0)   and +(0, 0)   .. (1, -1)
      .. controls +(0, -1)  and +(-1, 0)  .. (3, -2.4)
      .. controls +(0, 0)   and +(0, 0)   .. (17, -2.4)
      .. controls +(1, 0)   and +(0, -1)  .. (19, -1)
      .. controls +(0, 0)   and +(0, 0)   .. (19, 0)
      .. controls +(0, 1.5) and +(0, 1.5) .. (21, 0)
      .. controls +(0, 0)   and +(0, 0)   .. (21, -1)
      .. controls +(0, -2)  and +(0, -2)  .. (25, -1)
      .. controls +(0, 0)   and +(0, 0)   .. (25, 0)
      .. controls +(0, 1.5) and +(0, 1.5) .. (27, 0)
      .. controls +(0, 0)   and +(0, 0)   .. (27, -1)
      .. controls +(0, -2)  and +(0, -2)  .. (31, -1)
      .. controls +(0, 0)   and +(0, 0)   .. (31, 0)
      .. controls +(0, 1.5) and +(0, 1.5) .. (33, 0)
      .. controls +(0, 0)   and +(0, 0)   .. (33, -1)
      .. controls +(0, -3)  and +(3, 0)   .. (29, -5)
      .. controls +(0, 0)   and +(0, 0)   .. (3, -5)
      .. controls +(-3, 0)  and +(0, -3)   .. (-1, -1)
    ;
    \node[right=115pt of s] {$X$};
  \end{tikzpicture}
  \caption{%
    Legal cut $X \subseteq V$ corresponding to $\paren{1,0,3,2,1} \in {S_3}^5$.
    Vertices in each $U_i$ are $v_i^1, v_i^2, v_i^3$ from left to right.
  }
  \label{fig:example_of_legal_cut}
\end{figure}

For an $\paren{s,t}$-cut $X$ of $N$, let $\legalize{X}$ denote the legal $\paren{s,t}$-cut obtained by removing vertices in $X \cap U_i$ from $X$ for every $i \in \intset{n}$ with $\absprn{U_i \cap X} \geq 2$.
The \emph{capacity} $\app{c}{X}$ of $X$ is defined as sum of capacities $\app{c}{e}$ of all edges $e$ from $X$ to $V \setminus X$.
We say that a network $N$ \emph{represents} a function $\funcdoms{f}{{S_k}^n}{\setRext}$ if it satisfies the following conditions:
\begin{namedenum}{NR}
  \item There exists a constant $K \in \setR$ such that $\app{f}{x} = \app{c}{\app{\psi}{x}} + K$ for all $x \in {S_k}^n$.
  \item It holds $\app{c}{\legalize{X}} \leq \app{c}{X}$ for all $\paren{s,t}$-cuts $X$ of $N$.
\end{namedenum}
From (NR1), the minimum value of $f - K$ is equal to the capacity of a minimum $\paren{s,t}$-cut of $N$.
For every minimum $\paren{s,t}$-cut $X$ of $N$, $\legalize{X}$ is also a minimum $\paren{s,t}$-cut since $N$ satisfies the condition (NR2).
Therefore $\app{\psi^{-1}}{\legalize{X}}$ is a minimizer of $f$, and a minimum $\paren{s,t}$-cut can be computed by maximum flow algorithms.
Indeed, Iwata--Wahlstr\"{o}m--Yoshida~\cite{IwataY2016} showed that nonnegative combinations of basic $k$-submodular functions are representable by such networks; see Iwamasa~\cite{Iwamasa2017} for further study on this network construction.

Now we shall consider obtaining the PIP-representation for the minimizer set of a $k$-submodular function $\funcdoms{f}{{S_k}^n}{\setRext}$ represented by a network $N$.
The minimizer set of $f$ is isomorphic to the family of legal minimum $\paren{s,t}$-cuts of $N$ ordered by inclusion, where the isomorphism is $\psi$.
It is well-known that the family of (not necessarily legal) minimum $\paren{s,t}$-cuts forms a distributive lattice.
Thus, by Birkhoff representation theorem, the family is efficiently representable by a poset.
Picard--Queyranne~\cite{Picard1980} showed an algorithm to obtain the poset from the residual graph corresponding to a maximum $\paren{s,t}$-flow of $N$.
We describe their theorem briefly.
For an $\paren{s,t}$-flow $\phi$ of $N$, the \emph{residual graph} corresponding to $\phi$ is a directed graph $\paren{V, A_\phi}$, where
\begin{align*}
  A_\phi \defeq \setprnsep{a \in A}{\app{\phi}{a} < \app{c}{a}} \cup \setprnsep{\paren{u, v} \in V \times V}{\text{$\paren{v, u} \in A$ and $0 < \app{\phi}{v, u}$}}.
\end{align*}

\begin{theorem}[{\cite[Theorem~1]{Picard1980}}] \label{thm:picard}
  Let $N = \paren{V, A, c}$ be a directed network with $s, t \in V$ and $G$ the residual graph corresponding to a maximum $\paren{s, t}$-flow of $N$.
  Let $\Sigma$ be the set of strongly connected components (sccs) of $G$ other than the following:
  \begin{namedenum}{}
    \item Sccs reachable from $s$.
    \item Sccs reachable to $t$.
  \end{namedenum}
  Let $\leq$ be a partial order on $\Sigma$ defined by
  \begin{align*}
    \text{$X \leq Y$ if and only if $X$ is reachable from $Y$ on $G$}
  \end{align*}
  for every $X, Y \in \Sigma$.
  The ideal family of the poset $\paren{\Sigma, \leq}$ is isomorphic to the family of minimum $\paren{s,t}$-cuts of $N$ ordered by inclusion.
  The isomorphism $\tau$ is given by $\app{\tau}{I} \defeq X_0 \cup \paren{\bigcup_{X \in I} X}$, where $X_0$ is the set of vertices reachable from $s$.
\end{theorem}

Our result is the following.

\begin{theorem} \label{thm:network_algo}
  Let $N$ be a network representing a $k$-submodular function $\funcdoms{f}{{S_k}^n}{\setRext}$ and $G$ the residual graph corresponding to a maximum $\paren{s, t}$-flow of $N$.
  Let $\Sigma$ be the set of sccs of $G$ other than the following:
  \begin{namedenum}{}
    \item Sccs reachable from $s$.
    \item Sccs reachable to $t$.
    \item Sccs reachable to an scc containing two or more elements in $U_i$ for some $i \in \intset{n}$.
    \item Sccs reachable to sccs $X$ and $Y$ such that $X \neq Y$ and $\absprn{X \cap U_i} = \absprn{Y \cap U_i} = 1$ for some $i \in \intset{n}$.
  \end{namedenum}
  A partial order $\leq$ on $\Sigma$ is defined in the same way as Theorem~\ref{thm:picard}.
  Let $\inconsistent$ be a symmetric binary relation on $\Sigma$ defined as
  \begin{align*}
    & \text{$X \inconsistent Y$ if and only if there are distinct $X', Y' \in \Sigma$ such that} \\
    & \text{$X' \leq X, Y' \leq Y$ and $\absprn{X' \cap U_i} = \absprn{Y' \cap U_i} = 1$ for some $i \in \intset{n}$}.
  \end{align*}
  Then $\Sigma$ forms an elementary PIP with inconsistency relation $\inconsistent$.
  The consistent ideal family of $\Sigma$ is isomorphic to the minimizer set of $f$, where the isomorphism is $\psi^{-1} \circ \tau$.
\end{theorem}
\begin{proof}
  First we prove that $\Sigma$ is a PIP.
  We can see that $\inconsistent$ satisfies the condition (IC1) since for every $X, Y \in \Sigma$ with $X \inconsistent Y$, an scc $Z$ reachable to $X$ and $Y$ does not belong to $\Sigma$ according to the above exclusion rule (4).
  The condition (IC2) is also satisfied from the definition of the relation $\inconsistent$.
  Thus $\Sigma$ forms a PIP.
  
  Next we show $\app{\psi^{-1}}{\app{\tau}{I}} \in \minimizers{f}$ for every consistent ideal $I$ of $\Sigma$.
  Let $\Sigma'$ be the poset given in Theorem~\ref{thm:picard}.
  Note that $\Sigma$ is a subposet of $\Sigma'$.
  We show that $I$ is an ideal of $\Sigma'$.
  Suppose not.
  Then there exist $X \in I$ and $Y \in \Sigma' \setminus \Sigma$ such that $Y$ is reachable from $X$ and meets the above exclusion rules (3) or (4).
  Now since $X$ also satisfies the same exclusion rule, $X$ does not belong to $\Sigma$.
  This is a contradiction.
  Hence $I$ is an ideal of $\Sigma'$, and $\app{\tau}{I}$ is a minimum $\paren{s,t}$-cut (Theorem~\ref{thm:picard}).
  Moreover, from the exclusion rule (3) and the definition of $\inconsistent$, we can see that $\app{\tau}{I}$ is legal.
  Therefore $\app{\psi^{-1}}{\app{\tau}{I}}$ is a minimizer of $f$.
  
  Conversely, let $x \in {S_k}^n$ be a minimizer of $f$.
  Since $\app{\psi}{x}$ is a minimum $\paren{s,t}$-cut, $I \defeq \app{{\tau}^{-1}}{\app{\psi}{x}}$ is an ideal of $\Sigma'$ (Theorem~\ref{thm:picard}).
  Suppose that $I \nsubseteq \Sigma$.
  Then there exists $X \in I \setminus \Sigma$ which meets the exclusion rule (3) or (4).
  Suppose that $X$ meets the rule (3).
  Then $X$ is reachable to an scc $Y$ such that $\absprn{Y \cap U_i} \geq 2$ for some $i \in \intset{n}$.
  Now $Y$ is not reachable to $t$ since $X$ is not reachable to $t$.
  Thus $Y$ meets the rule (1) or belongs to $I$ otherwise.
  In either case it holds $Y \subseteq \app{\psi}{x}$.
  This contradicts the fact that $\app{\psi}{x}$ is legal.
  A similar argument can also be applied in the case where $X$ meets the rule (4).
  Therefore $I \subseteq \Sigma$ holds, and $I$ is an ideal of $\Sigma$ since $\Sigma$ is a subposet of $\Sigma'$.
  The consistency of $I$ is an immediate consequence of the fact that $\app{\psi}{x}$ is legal.
  
  Now we have shown that $\psi^{-1} \circ \tau$ is a bijection from $\consideals{\Sigma}$ to $\minimizers{f}$.
  In addition, $\psi^{-1} \circ \tau$ clearly preserves the orders, hence it is an isomorphism.
  Finally from Corollary~\ref{cor:elementary_pip}, $\Sigma$ is elementary.
\end{proof}

Algorithm~\ref{alg:remove_sccs} shows a procedure to obtain $\Sigma$ from the residual graph $G$.
First we can obtain the sccs of $G$ in $\Order{kn + \tilde{m}}$ time, where $\tilde{m} \defeq \absprn{A}$.
Additionally, the exclusion rules (1), (2) and (3) can be applied to the sccs in the same time complexity.
Hence it is only the exclusion rule (4) that we should carefully take account of.
An efficient way is described in Line 4 to 10 in Algorithm~\ref{alg:remove_sccs}.
For each scc $X$, the algorithm memorizes the set $U_X$ of vertices reachable to $X$.
Now since the size of each $U_X$ is $\Order{n}$ at any moment, Algorithm~\ref{alg:remove_sccs} runs in $\Order{\absprn{V} + n\absprn{A}} = \Order{kn + n\tilde{m}}$ time.
Therefore the time complexity for obtaining $\Sigma$ from $G$ is much less than the one for computing $G$ from the network $N$.
Consequently, we obtain the following theorem:

\begin{algorithm}[tp]
  \caption{Obtain sccs which do not meet the exclusion rules}
  \label{alg:remove_sccs}
  \begin{algorithmic}[1]
    \Input  The residual graph $G = \paren{V, A_\phi}$ corresponding to a maximum $\paren{s,t}$-flow $\phi$
    \Output The set $\Sigma$ of sccs of $G$ defined in Theorem~\ref{thm:network_algo}
    \Function{ApplyExclusionRules}{$G$}
      \State{$\Sigma \gets \text{the set of sccs of $G$}$}
      \State{Remove all sccs from $\Sigma$ which meet the exclusion rules (1), (2) or (3)}
      \For{$X \in \Sigma$ in the reverse topological order of $G$}
        \State{$U_X \gets X$}
        \State{$\mathcal{Y} \defeq \setprnsep{Y \in \Sigma}{\text{there is an edge $\paren{x, y} \in A_\phi$ for some $x \in X$ and $y \in Y$}}$}
        \For{$Y \in \mathcal{Y}$}
          \State{$U_X \gets U_X \cup U_Y$}
          \If{$\absprn{U_X \cap U_i} \geq 2$ for some $i \in \intset{n}$}
            \State{Remove all sccs from $\Sigma$ which are reachable to $X$, and go to Line 4}
          \EndIf
        \EndFor
      \EndFor
      \State{\Return $\Sigma$}
    \EndFunction
  \end{algorithmic}
\end{algorithm}

\begin{theorem} \label{thm:network_time}
  Let $\funcdoms{f}{{S_k}^n}{\setRext}$ be a $k$-submodular function represented by a network $N$ with $\tilde{m}$ edges.
  The PIP-representation for the minimizer set of $f$ is obtained in $\Order{\maxflow{kn}{\tilde{m}}}$ time.
\end{theorem}

\subsection{Potts $k$-submodular functions}
\label{sec:potts}

Here we consider a practically important subclass of network representable $k$-submodular functions, called \emph{Potts $k$-submodular functions}. 
Let $\paren{V, E}$ be a connected undirected graph on vertex set $V = \intset{n}$ with $m = \absprn{E}$, where each edge $\setprn{i, j} \in E$ has a positive edge weight $\lambda_{i, j}$.
Let $\intset{k}$ be the set of labels.
A Potts $k$-submodular function is a $k$-submodular function $\funcdoms{\tilde{g}}{{S_k}^n}{\setR}$ of the following form:
\begin{align} \label{eqn:tilde_g}
  \app{\tilde{g}}{x} = \sum_{i=1}^n \app{\tilde{g_i}}{x_i} + \sum_{\setprn{i, j} \in E} \lambda_{i, j} \app{d}{x_i, x_j}
  \quad
  \paren{x \in {S_k}^n},
\end{align}
where $g_i$ is any $k$-submodular function on $S_k$ for each $i \in \intset{n}$ and $d$ is a $k$-submodular function on ${S_k}^2$ defined by
\begin{align*}
  \app{d}{a, b} \defeq \begin{cases}
    1   & \paren{0 \neq a \neq b \neq 0}, \\
    0   & \paren{a = b}, \\
    1/2 & \mathrm{otherwise}
  \end{cases}
\end{align*}
for each $a,b \in S_k$.
A Potts $k$-submodular function is naturally associated with a \emph{Potts energy function} $\funcdoms{g}{\intset{k}^n}{\setR}$:
\begin{align} \label{eq:potts}
  \app{g}{x} = \sum_{i=1}^n \app{g_i}{x_i} + \sum_{\setprn{i, j} \in E} \lambda_{i, j} \app{1_{\neq}}{x_i, x_j}
  \quad \paren{x \in \intset{k}^n},
\end{align}
where $g_i$ is any function on $\intset{k}$ for each $i \in \intset{n}$ and $\funcdoms{1_{\neq}}{\intset{k}^2}{\setR}$ is defined by $\app{1_{\neq}}{\alpha, \beta} \defeq 1$ if $\alpha \neq \beta$ and $\app{1_{\neq}}{\alpha, \beta} \defeq 0$ if $\alpha = \beta$.

Finding a labeling $x \in \intset{k}^n$ of the minimum Potts energy is NP-hard for $k \geq 3$ but particularly important in computer vision applications.
Useful information of optimal labelings of the Potts energy can be extracted from a minimizer of a Potts $k$-submodular function with appropriate $k$-submodular functions $\tilde{g}_i$.
Define each $\tilde{g}_i$ by $\app{\tilde{g}_i}{\alpha} \defeq \app{g_i}{\alpha}$ for $\alpha \in \intset{k}$ and $\app{\tilde{g}_i}{0} \defeq \min_\condition{\beta, \gamma \in \intset{k}}{\beta \neq \gamma} \paren{\app{g_i}{\beta} + \app{g_i}{\gamma}}/2$.
In this case, $\tilde{g}$ is a \emph{$k$-submodular relaxation} of $g$, 
and an optimal labeling of $g$ is a partially recovered from a minimizer of $g$; see the next section.
Another choice of $\tilde{g_i}$ is: $\app{\tilde{g_i}}{\alpha} \defeq \paren{\app{g_i}{\alpha} - \min_{\beta \in \intset{k} \setminus \setprn{\alpha}} \app{g_i}{\beta}}/2$ for $\alpha \in \intset{k}$ and $\app{\tilde{g}}{0} \defeq 0$.
Also in this case, a part of an optimal labeling is obtained from a minimizer of $\tilde{g}$, and coincides with Kovtun's partial labeling~\cite{Gridchyn2013, Kovtun2003}.
 
The goal of this section is to develop a fast algorithm to construct the PIP of a Potts $k$-submodular function $\tilde{g}$.
Notice that $\tilde{g}$ is network-representable with $km$ edges~\cite{IwataY2016}.
Therefore we can obtain a minimizer as well as the PIP-representation for $\tilde{g}$ in $\Order{\maxflow{kn}{km}}$ time by the network construction in the previous section.
However it is hard to apply this algorithm to the vision application with large $k \paren{\sim 60}$ in~\cite{Gridchyn2013}.
Gridchyn--Kolmogorov~\cite{Gridchyn2013} developed an $\Order{\log k \cdot \maxflow{n}{m}}$-time algorithm to find a minimizer of $\tilde{g}$.
The main theorem in this section is a stronger result that the PIP-representation is also obtained in the same time complexity.
\begin{theorem} \label{thm:Potts}
  The PIP-representation for the minimizer set of $\tilde{g}$ is obtained in\\
  $\Order{\log k \cdot \maxflow{n}{m}}$ time.
\end{theorem} 
The rest of this subsection is devoted to proving this theorem. 
First we construct a network $N$, different from the one in the previous section.
For each $i \in \intset{n}$, decompose $\tilde{g}_i$ as follows.
Let $\funcdoms{1_{=}}{{S_k}^2}{\setR}$ be defined by $\app{1_{=}}{a, b} \defeq 1$ if $a = b$ and $\app{1_{=}}{a, b} \defeq 0$ otherwise.
Choose a minimizer $\gamma_i \in S_k$ of $\tilde{g}_i$.
Then  $\tilde{g}_i$ is represented as
\begin{align*}
  \app{\tilde{g}_i}{x_i}
    = \app{\tilde{g}_i}{\gamma_i}
    + \mu_i \app{d}{\gamma_i, x_i}
    + \sum_{\alpha \in \intset{k} \setminus \setprn{\gamma_i}} \sigma_{i, \alpha} \app{1_{=}}{\alpha, x_i},
\end{align*}
where $\mu_i \defeq 2\paren{\app{\tilde{g}_i}{0} - \app{\tilde{g}_i}{\gamma_i}} \, \paren{\geq 0}$ and
$\sigma_{i, \alpha} \defeq \app{\tilde{g}_i}{\alpha} - 2 \app{\tilde{g}_i}{0} + \app{\tilde{g}_i}{\gamma_i}$ for $\alpha \in \intset{k}$.
We remark that $\sigma_{i, \alpha}$ is nonnegative by $k$-submodularity, and that $\mu_i > 0$ implies $\gamma_i \neq 0$.

Let us construct $N$.
Starting from $\paren{V, E}$, define the edge-capacity $\app{c}{\setprn{i, j}}$ of each edge $\setprn{i, j} \in E$ by $\lambda_{i, j}$.
Next add new vertices $s_1, s_2, \ldots, s_k$, called \emph{terminals}.
For each $i \in \intset{n}$, if $\mu_i > 0$ with $\alpha = \gamma_i \in \intset{k}$, add a new edge $\setprn{i, s_\alpha}$ of capacity $\app{c}{\setprn{i, s_\alpha}} \defeq \mu_i$.
An edge $\setprn{i, s_\alpha}$ is called a \emph{terminal edge}.
For each $\alpha \in \intset{k}$ with $\sigma_{i, \alpha} > 0$, add a new vertex $i^\alpha$ and a new edge $\setprn{i, i^\alpha}$ of capacity $\app{c}{\setprn{i, i^\alpha}} \defeq 2 \sigma_{i, \alpha}$.
A vertex $i^\alpha$ is called the \emph{$\alpha$-fringe} of $i$. 
Let $S \defeq \setprn{s_1, s_2, \ldots, s_k}$. 
Let $V_0$ be the set of all fringes, $E_0$ the set of all edges incident to fringes, and $E_S$ the set of all terminal edges.
Let $\tilde{V} \defeq V \cup V_0 \cup S$ and $\tilde{E} \defeq E \cup E_0 \cup E_S$.
Let $N = \paren{\tilde{V}, \tilde{E}, c}$ be the resulting network.

Second we show that $\tilde{g}$ is represented as a certain multicut function in $N$.
For a vertex subset $X$, the cut capacity $\app{c}{X}$ of $X$ is the sum of $\app{c}{e}$ of all edges $e$ between $X$ and $V \setminus X$.
For $\alpha \in \intset{k}$, a vertex subset $X$ is called an \emph{$s_\alpha$-isolating cut} or \emph{$\alpha$-cut} if $s_\alpha \in X$, $s_\beta \not \in X$ for $\beta \in \intset{k} \setminus \setprn{\alpha}$, and $X$ contains no $\alpha$-fringe.
A \emph{semi-multicut} is an ordered partition $\paren{X_0, X_1, \ldots, X_k}$ of $\tilde{V}$ such that $X_\alpha$ is an $\alpha$-cut for each $\alpha \in \intset{k}$.
The \emph{capacity} $\app{c}{\mathcal{X}}$ of a semi-multicut $\mathcal{X} = \paren{X_0, X_1, \ldots, X_k}$ is defined by
\begin{align*}
  \app{c}{\mathcal{X}} \defeq \frac{1}{2} \sum_{\alpha \in \intset{k}} \app{c}{X_\alpha}.
\end{align*}
An \emph{admissible} semi-multicut is a semi-multicut $\paren{X_0, X_1, X_2, \ldots, X_k}$ such that for each $\alpha \in \intset{k}$, each $\alpha$-fringe $i^\alpha$ belongs to $X_0$ if $i \in X_0 \cup X_\alpha$ and belongs to $X_\beta$ if $i \in X_\beta$ for $\beta \in \intset{k} \setminus \setprn{\alpha}$.
Observe that the part to which a fringe of $i \in \intset{n}$ belongs is uniquely determined from the part which $i$ belongs.
For an admissible semi-multicut $\mathcal{X} = \paren{X_0, X_1, \ldots, X_k}$, define $\app{x}{\mathcal{X}} = \paren{\app{x_1}{\mathcal{X}}, \app{x_2}{\mathcal{X}}, \ldots, \app{x_n}{\mathcal{X}}} \in {S_k}^n$ by $\app{x_i}{\mathcal{X}} \defeq a \in S_k$ if and only if $i \in X_a$.
This map $\mathcal{X} \mapsto \app{x}{\mathcal{X}}$ is a bijection from the family of all admissible semi-multicuts to ${S_k}^n$; see Figure~\ref{fig:example_of_admissible_semimultiway_cut}.
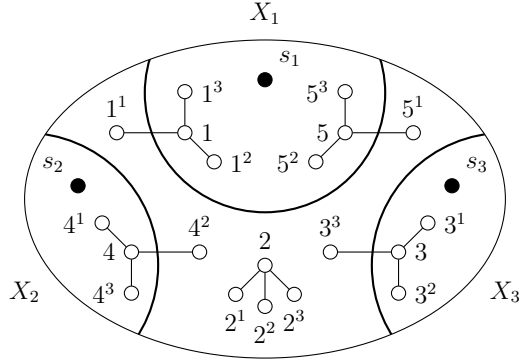
\begin{figure}[t]
  \centering
  \begin{tikzpicture}[every node/.style={scale=0.8}, x=5pt, y=5pt]
    \node at (0, 70pt) {$X_1$};
    \node at (-90pt, -35pt) {$X_2$};
    \node at ( 90pt, -35pt) {$X_3$};
    \begin{scope}
      \tikzstyle{vertex}=[circle, draw, scale=0.6];
      \draw[name path=oval] ellipse (90pt and 60pt);
      \clip (0,0) ellipse (90pt and 60pt);
      \draw[name path=x1, thick] (  0pt,  40pt) circle (45pt);
      \draw[name path=x2, thick] (-90pt, -25pt) circle (50pt);
      \draw[name path=x3, thick] ( 90pt, -25pt) circle (50pt);
      \node[vertex, label=right:1]                            (1) at (-30pt, 25pt) {};
      \node[vertex,        left=20pt of 1, label=above:$1^1$] (11) {};
      \node[vertex, below right=10pt of 1, label=right:$1^2$] (12) {};
      \node[vertex,       above=10pt of 1, label=right:$1^3$] (13) {};
      \draw (1) -- (11);
      \draw (1) -- (12);
      \draw (1) -- (13);
      \node[vertex, label=above:2]                           (2) at (0pt, -25pt) {};
      \node[vertex, below left=10pt of 2, label=below:$2^1$] (21) {};
      \node[vertex, below=10pt of 2, label=below:$2^2$]      (22) {};
      \node[vertex, below right=10pt of 2, label=below:$2^3$]     (23) {};
      \draw (2) -- (21);
      \draw (2) -- (22);
      \draw (2) -- (23);
      \node[vertex, label=right:3]                            (3) at (50pt, -20pt) {};
      \node[vertex, above right=10pt of 3, label=right:$3^1$] (31) {};
      \node[vertex, below=10pt of 3, label=right:$3^2$]       (32) {};
      \node[vertex, left=20pt of 3, label=above:$3^3$]        (33) {};
      \draw (3) -- (31);
      \draw (3) -- (32);
      \draw (3) -- (33);
      \node[vertex, label=left:4]                           (4) at (-50pt, -20pt) {};
      \node[vertex, above left=10pt of 4, label=left:$4^1$] (41) {};
      \node[vertex, right=20pt of 4, label=above:$4^2$]     (42) {};
      \node[vertex, below=10pt of 4, label=left:$4^3$]      (43) {};
      \draw (4) -- (41);
      \draw (4) -- (42);
      \draw (4) -- (43);
      \node[vertex, label=left:5]                            (5) at (30pt, 25pt) {};
      \node[vertex,      right=20pt of 5, label=above:$5^1$] (51) {};
      \node[vertex, below left=10pt of 5, label=left:$5^2$]  (52) {};
      \node[vertex,      above=10pt of 5, label=left:$5^3$]  (53) {};
      \draw (5) -- (51);
      \draw (5) -- (52);
      \draw (5) -- (53);
      \node[vertex, fill, label=above right:$s_1$] (s1) at (  0pt, 45pt) {};
      \node[vertex, fill, label=above left:$s_2$]  (s3) at (-70pt,  5pt) {};
      \node[vertex, fill, label=above right:$s_3$]  (s3) at ( 70pt,  5pt) {};
    \end{scope}
  \end{tikzpicture}
  \caption{Admissible semi-multiway cut $\mathcal{X} = \paren{X_0, X_1, X_2, X_3}$ corresponding to $\paren{1,0,3,2,1} \in {S_3}^5$.}
  \label{fig:example_of_admissible_semimultiway_cut}
\end{figure}
The next lemma says that a $k$-submodular function $\tilde{g}$ is actually represented by capacities of admissible semi-multicuts.
\begin{lemma}
  For any admissible semi-multicut $\mathcal{X}$, it holds
  \begin{align*}
    \app{\tilde{g}}{\app{x}{\mathcal{X}}}
      = \app{c}{\mathcal{X}} 
      + \sum_{i \in \intset{n}} \app{\tilde{g}_i}{\gamma_i}.
  \end{align*}
\end{lemma}
\begin{proof}
  Let $x \defeq \app{x}{\mathcal{X}}$.
  The capacity $2 \sigma_{i, \alpha}$ of a fringe edge $\setprn{i, i^\alpha}$ contributes to $\app{c}{\mathcal{X}}$ by $\sigma_{i, \alpha}$ if $i \in X_\alpha$ and by $0$ otherwise.
  Thus the contribution is equal to $\sigma_{i, \alpha} \app{1_{=}}{\alpha, x_i}$.
  The capacity $\mu_i$ of a terminal edge $\setprn{i, s_\alpha}$ contributes to $\app{c}{\mathcal{X}}$ by $0$ if $i \in X_\alpha$ and by $\mu_i / 2$ if $i \in X_0$, and by $\mu_i$ if $i \in X_\beta$ with $\beta \neq \alpha$.
  Thus the contribution is equal to $\mu_i \app{d}{x_i, \alpha}$.
  Similarly, we verify that the contribution of the capacity $\lambda_{i, j}$ of $\setprn{i, j} \in E$ is equal to $\lambda_{i, j} \app{d}{x_i, x_j}$.
  Thus the claimed equality holds.
\end{proof}

Third we show that a minimum admissible semi-multicut is easily obtained by $k$ max-flow computations, where ``minimum'' is with regard to the cut capacity. 
An \emph{admissible} $\alpha$-cut is an $\alpha$-cut $X$ such that for each $\beta \in \intset{k} \setminus \setprn{\alpha}$, each $\beta$-fringe $i^\beta$ belongs to $X$ if $i \in X$ and $\tilde{V} \setminus X$ otherwise.
Then an admissible semi-multicut $\paren{X_0, X_1, X_2, \ldots, X_k}$ is exactly a partition of $\tilde{V}$ such that $X_\alpha$ is an admissible $\alpha$-cut for each $\alpha \in \intset{k}$.
\begin{lemma} \label{lem:minimum_semimultiway_cut}
  \begin{enumerate}
    \item Any minimum $\alpha$-cut is admissible.
    \item For $\alpha \in \intset{k}$, let $Y_\alpha$ be the inclusion-minimal minimum $\alpha$-cut.
          Then $\paren{Y_0, Y_1, \ldots, Y_k}$ is a minimum admissible semi-multicut.
  \end{enumerate}
  In particular, a minimum admissible semi-multicut is exactly a partition $(X_0, X_1, \ldots, X_k)$ of $\tilde{V}$ such that $X_\alpha$ is a minimum $\alpha$-cut for each $\alpha \in \intset{k}$.
\end{lemma}
\begin{proof}
  (1).
  Let $X$ be an $\alpha$-cut.
  For $\beta \in \intset{k} \setminus \setprn{\alpha}$, if the $\beta$-fringe $i^\beta$ of $i \in X$ is outside of $X$, then include $i^\beta$ into $X$ to decrease the cut capacity.
  Similarly, for $\beta \in \intset{k} \setminus \setprn{\alpha}$, if the $\beta$-fringe $i^\beta$ of $i \in V \setminus X$ belongs to $X$, then remove $i^\beta$ from $X$ to decrease the cut capacity.
  (2) is immediate from the standard uncrossing argument; see the proof of Lemma~\ref{lem:uncrossing}.
\end{proof}
In particular, a minimum admissible multicut is obtained by computing a minimal minimum $\alpha$-cut for each $\alpha \in \intset{k}$.
The network $N$ has at most $k + n + nk$ vertices and $m + n+ nk$ edges.
When computing a minimum $\alpha$-cut, all $\beta$-fringes with $\beta \neq \alpha$ can be removed, and a max-flow algorithm is performed on a network of $n+2$ vertices and $m + 2n$ edges
(after combining $s_\beta$ for $\beta \neq \alpha$ and all $\alpha$-fringes into a single vertex).
Thus we obtain the following, which was essentially shown in~\cite{Gridchyn2013, Kovtun2003}.
\begin{lemma}[\cite{Gridchyn2013, Kovtun2003}]
  A minimizer of $\tilde{g}$ can be obtained in $\Order{k \maxflow{n}{m}}$ time.
\end{lemma}

Fourth we explain how to obtain the PIP representation from maximum $\alpha$-flows for $\alpha \in \intset{k}$, where by an \emph{$\alpha$-flow} we mean a flow from $s_\alpha$ to the union of 
$S \setminus \setprn{s_\alpha}$ and $\alpha$-fringes.
To construct the PIP, we use the following intersecting properties of minimum isolating cuts. 
Here a minimum $\alpha$-cut is simply called an {\em $\alpha$-mincut}.

\begin{lemma} \label{lem:uncrossing}
  \begin{enumerate}
    \item  For distinct $\alpha, \beta \in \intset{k}$, if $X$ is an $\alpha$-mincut and $Y$ is a $\beta$-mincut, then $X \setminus Y$ is an $\alpha$-mincut and $Y \setminus X$ is a $\beta$-mincut.
    \item For distinct $\alpha, \beta, \gamma \in \intset{k}$, if $X$ is an $\alpha$-mincut, $Y$ is a $\beta$-mincut and $Z$ is a $\gamma$-mincut, then $X \cap Y \cap Z = \varnothing$.  
  \end{enumerate}
\end{lemma}
\begin{proof}
  In the theory of minimum cuts on undirected networks, the following inequalities are well-known:
  \begin{align*}
    \app{c}{X} + \app{c}{Y} & \geq \app{c}{X \setminus Y} + \app{c}{Y \setminus X}, \\
    \app{c}{X} + \app{c}{Y} + \app{c}{Z} & \geq \app{c}{X \setminus \paren{Y \cup Z}} + \app{c}{Y \setminus \paren{Z \cup X}} + \app{c}{Z \setminus \paren{X \cup Y}} + \app{c}{X \cap Y \cap Z}
  \end{align*}
  for every $X, Y, Z \subseteq \tilde{V}$.
  Then (1) is an immediate consequence of the first inequality and the fact that any subset of an $\alpha$-cut containing $s_\alpha$ is again an $\alpha$-cut.
  (2) is also immediate from the second inequality and the condition that
  $\paren{V, E}$ is connected and each edge of $N$ has a positive capacity.
\end{proof}

By Lemma~\ref{lem:uncrossing}~(2), each vertex belongs to at most two minimum isolating cuts.
Let $\overrightarrow{N}$ denote the directed network obtained from $N$ by replacing each undirected edge $\setprn{u, v} \in \tilde E$ by two directed edge $\paren{u, v}$ and $\paren{v, u}$ of capacity $\app{c}{u, v} = \app{c}{v, u} \defeq \app{c}{\setprn{u, v}}$.
For each $\alpha \in \intset{k}$, consider the network obtained from $\overrightarrow{N}$
by removing all $\beta$-fringes with $\beta \neq \alpha$ and contracting terminals $s_\beta$ with $\beta \neq \alpha$ and $\alpha$-fringes into a single terminal $s'$, and consider
the residual graph $G_\alpha$ corresponding to a maximum $\paren{s_\alpha, s'}$-flow in this network.
Let $\Sigma_\alpha = \paren{\Sigma_\alpha, \leq_\alpha}$ be the poset obtained from $G_\alpha$ in the same way as defined in Theorem~\ref{thm:picard}.
Here each element in $\Sigma_\alpha$ is a subset of $V$ (not including terminals and fringes).
The ideal family of each $\Sigma_\alpha$ is isomorphic to 
the family of minimum $\alpha$-cuts.
The intersecting part in $\Sigma_\alpha$ and $\Sigma_\beta$ is described as follows.
\begin{lemma} \label{lem:intersection_of_isolating_cuts}
  Let $\alpha, \beta \in \intset{k}$ with $\alpha \neq \beta$.
  The following hold:
  \begin{enumerate}
    \item For every $A \in \Sigma_\alpha$ and $B \in \Sigma_\beta$, it holds either $A \cap B = \varnothing$ or $A = B$.
    \item For every $A, B \in \Sigma_\alpha \cap \Sigma_\beta$, if $A \leq_\alpha B$, then it holds $B \leq_\beta A$.
  \end{enumerate}
\end{lemma}
\begin{proof}
  (1).
  Assume that $A \cap B \neq \varnothing$ and $A \neq B$.
  We can assume $A \setminus B \neq \varnothing$.
  Consider an $\alpha$-mincut $X$ with $A \subseteq X$ and consider a $\beta$-mincut $Y$ with $B \subseteq Y$.
  Take $Y$ minimal.
  Then every scc of $G_\beta$ included in $Y$ is less than or equal to $B$ with respect to $\leq_\beta$.
  Suppose that $Y$ contains $A$.
  There is a $\beta$-mincut $Z$ containing $A \setminus B$ and disjoint from $B$.
  By Lemma~\ref{lem:uncrossing}~(1), $X \setminus Z$ is an $\alpha$-mincut and properly intersects $A$.
  However this is impossible since the set of non-fringe vertices in each $\alpha$-mincut is a disjoint union of sccs of $G_\alpha$.
  Suppose that $Y$ does not contain $A$.
  Then $X \setminus Y$ is an $\alpha$-mincut and properly intersects $A$ again.
  This is a contradiction.
  
  (2).
  Assume that $A \leq_\alpha B$ and $B \not\leq_\beta A$.
  There is a $\beta$-mincut $Y$ containing $A$ and disjoint with $B$.
  Consider an $\alpha$-mincut $X$ containing $B$. 
  Then $X \setminus Y$ is an $\alpha$-mincut, contains $B$, and does not contain $A$.
  However this contradicts the assumption that $A \leq_\alpha B$.
\end{proof}

By Lemma~\ref{lem:intersection_of_isolating_cuts}, we obtain the elementary PIP representing minimum admissible cuts just by ``gluing'' $\Sigma_1, \Sigma_2, \ldots, \Sigma_k$ along the intersections.
Let $P \defeq \bigcup_{\alpha \in \intset{k}} \Sigma_\alpha \times \setprn{\alpha}$ and $\leq$ a partial order on $P$ defined by
\begin{align*}
  \text{$\paren{X, \alpha} \leq \paren{Y, \beta}$ if and only if $\alpha = \beta$ and $X \leq_\alpha Y$}
\end{align*}
for every $\paren{X, \alpha}, \paren{Y, \beta} \in P$.
In addition, let $\mininconsistent$ be a symmetric binary relation on $P$ defined by
\begin{align*}
  \text{$\paren{X, \alpha} \mininconsistent \paren{Y, \beta} $ if and only if $\alpha \neq \beta$ and $X = Y$}
\end{align*}
for every $\paren{X, \alpha}, \paren{Y, \beta} \in \Sigma$.

\begin{theorem} \label{thm:PIP_Potts}
  The triplet $P = \paren{P, \leq, \mininconsistent}$ 
  is an elementary PIP with minimal inconsistency relation $\mininconsistent$.
  The consistent ideal family $\consideals{P}$ and the family of minimum admissible semi-multicuts of $N$ are in one-to-one correspondence by the map
  \begin{align*}
    I \mapsto \paren{X_0^I, X^I_1, X^I_2, \ldots, X^I_k}, 
  \end{align*}
  where, for each $\alpha \in \intset{k}$, $X_\alpha^I$ is the admissible $\alpha$-cut containing all vertices $i \in V$ such that $i$ is reachable from $s_\alpha$ in $G_\alpha$ or belongs to $X$ for some $\paren{X, \alpha} \in I$.
\end{theorem}
\begin{proof}
  It is easy to see from Lemma~\ref{lem:intersection_of_isolating_cuts}~(2) that $P$ is a PIP with minimal inconsistency relation $\mininconsistent$.
  
  We next show that $\consideals{P}$ represents the family of minimum semi-multiway cuts. 
  Let $I \in \consideals{P}$.
  Then $X_\alpha^I$ is a minimum $s_\alpha$-isolating cut (by Theorem~\ref{thm:picard}).
  By consistency, it necessarily holds $X_\alpha^I \cap X_\beta^I = \emptyset$ 
  for $\alpha \neq \beta$.
  Thus $\paren{X^I_0, X^I_1, X^I_2, \ldots, X^I_k}$ is a minimum semi-multiway cut.
  
  Conversely, let $\paren{X_0, X_1, X_2, \ldots, X_k}$ be a minimum semi-multiway cut.
  Each $X_\alpha$ is a minimum $s_\alpha$-isolating cut, and is represented by an ideal $I_\alpha$ of $\Sigma_\alpha$.
  Now $I \defeq \bigcup_{\alpha \in \intset{k}} I_\alpha \times \setprn{\alpha}$ is a consistent ideal of $P$ since $X_1, X_2, \ldots, X_k$ are pairwise disjoint.
  Then it holds $X_\alpha = X_\alpha^I$ for $\alpha \in \intset{k}$.
  
  Now $P$ represents a $\paren{\sqcap, \sqcup}$-closed set in ${S_k}^n$, and is necessarily elementary by Corollary~\ref{cor:elementary_pip}.
\end{proof}

Therefore the PIP-representation is obtained by computing a maximum $\alpha$-flow for each $\alpha \in \intset{k}$.

\begin{lemma}
  The PIP-representation for the minimizer set of $\tilde{g}$ is obtained in $\mathrm{O}(k \maxflow{n}{m})$ time.
\end{lemma}

Finally we present an improved algorithm of time complexity $\Order{\log k \cdot \maxflow{n}{m}}$.
The key is the existence of a single ``multiflow'' that includes all maximum $\alpha$-flows.
Let $\mathcal{Q}$ denote the set of $\alpha$-paths over all $\alpha \in \intset{k}$, where an \emph{$\alpha$-path} is a path connecting $s_\alpha$ and an $\alpha$-fringe or a terminal $s_\beta$ with $\beta \neq \alpha$.
A \emph{multiflow} is a nonnegative-valued function $f$ on $\mathcal{Q}$ satisfying the capacity constraint:
\begin{align*}
  \app{f}{e} \defeq \sum \setprnsep{\app{f}{Q}}{Q \in \mathcal{Q}: \text{$Q$ contains $e$}} \leq \app{c}{e} \quad \paren{e \in \tilde{E}}.
\end{align*}
Let $\absprn{f}$ denote the total-flow value of $f$:
\begin{align*}
  \absprn{f} \defeq \sum \setprnsep{\app{f}{Q}}{Q \in \mathcal{Q}}.
\end{align*}
For $\alpha \in \intset{k}$, let $f_\alpha$ be the submultiflow of $f$ defined by $\app{f_\alpha}{Q} \defeq \app{f}{Q}$ if $Q$ is an $\alpha$-path and $\app{f_\alpha}{Q} \defeq 0$ otherwise.
Although the set $\mathcal{Q}$ is exponential, we can efficiently handle multiflow $f$ by keeping $f$ as $k$ flows of node-arc form in $\overrightarrow{N}$, as in \cite[p. 65--66]{Ibaraki1998}.
The following is a special case of~\cite[Theorem 1.2]{Hirai2010} (a version of \emph{multiflow locking theorem}).
\begin{lemma}[\cite{Hirai2010}]
  There exists a multiflow $f$ such that $\absprn{f_\alpha}$ is equal to the minimum capacity of an $\alpha$-cut for each $\alpha \in \intset{k}$. 
\end{lemma}  
Thus the submultiflow $f_\alpha$ of $f$ turns into a maximum $\alpha$-flow in $\overrightarrow{N}$.
We call such a multiflow \emph{locking}.
Our goal is to show that a locking multiflow $f$ is obtained in $\Order{\log k \cdot \maxflow{n}{m}}$ time.
This, consequently, yields $\Order{\log k \cdot \maxflow{n}{m}}$-time algorithm 
to obtain posets $\Sigma_1, \Sigma_2, \ldots, \Sigma_k$ and the desired PIP $P = \bigcup_{\alpha \in \intset{k}} \Sigma_\alpha \times \setprn{\alpha}$.

In the case where there are no fringes, the problem of finding a locking multiflow is nothing but the \emph{maximum free multiflow problem}, which is a well-studied problem in multiflow theory.
Ibaraki--Karzanov--Nagamochi~\cite{Ibaraki1998} developed an $\Order{\log k \cdot \maxflow{n}{m}}$-time algorithm (\emph{IKN-algorithm}) to obtain a locking multiflow.
Babenko--Karzanov~\cite{Babenko2012} extended the IKN-algorithm to a more general case, and can be applied to our case. 
For completeness, we present a direct adaptation of IKN-algorithm to the case where fringes exist,  though our algorithm may be regarded as a specialization of~\cite{Babenko2012}. 
The pseudo code is shown in Algorithm~\ref{alg:locking}.

\begin{algorithm}[tp]
  \caption{Compute a locking multiflow}
  \label{alg:locking}
  \begin{algorithmic}[1]
    \Input  A network $N$ with terminal set $S$
    \Output A locking multiflow in $N$
    \Function{Locking}{$N$}
      \If{$\absprn{S} \geq 4$}
        \State{Divide $S$ into $S'$ and $S''$ with $\absprn{S'} = \floor{\absprn{S}/2}$ and $\absprn{S''} = \ceil{\absprn{S}/2}$}
        \State{Compute a minimum cut $X$ with $S' \subseteq X$ and $X \cap S'' = \varnothing$}
        \State{Construct two networks $N'$ and $N''$}
        \State{$f' \defeq \Call{Locking}{N'}$ and $f'' \defeq \Call{Locking}{N''}$}
        \State{Aggregate $f'$ and $f''$ into a locking multiflow $f$ in $N$}
      \Else
        \If{there is an $\alpha$-fringe}
          \State{Compute a minimum $\alpha$-cut $X$}
          \State{Construct two network $N'$ and $N''$}
          \State{$f' \defeq \text{a maximum $\paren{s_\alpha, s'}$-flow}$ and $f'' \defeq \Call{Locking}{N''}$}
          \State{Aggregate $f'$ and $f''$ to obtain a locking multiflow $f$ in $N$}
        \Else
          \State{Compute a locking multiflow $f$ by IKN-algorithm}
        \EndIf
      \EndIf
      \State{\Return $f$}
    \EndFunction
  \end{algorithmic}
\end{algorithm}

Let us explain the detail of the algorithm.
Consider the case $\absprn{S} \geq 4$. 
As in IKN-algorithm, our algorithm divides terminal set $S$ into two sets $S'$ and $S''$ such that $\absprn{S'} = \floor{\absprn{S}/2}$ and $\absprn{S''} = \ceil{\absprn{S}/2}$, and find a minimum cut $X$ with $S' \subseteq X$ and $S'' \cap X = \varnothing$.
Here fringes may be removed in computation since $i \in X$ implies that all fringes of $i$ belong to $X$.
Two networks $N'$ and $N''$ are constructed as follows.
The network $N'$ is obtained from $N$ by contracting $\tilde{V} \setminus X$ into a single terminal $s'$ and by removing all $\alpha$-fringes for $\alpha \in S''$.  
Similarly, the network $N''$ is obtained from $N$ by contracting $X$ into a single terminal $s''$ and by removing all $\alpha$-fringes for $\alpha \in S'$. 

Suppose that we have a locking multiflow $f'$ in $N'$ and locking multiflow $f''$ in $N''$.
Then a locking multiflow $f$ in $N$ is obtained by ``aggregating'' $f'$ and $f''$ as follows.
An $\alpha$-path $Q$ in $N'$ not connecting $s'$ is regarded as an $\alpha$-path in $N$.  
Set $\app{f}{Q} \defeq \app{f'}{Q}$ for such a path $Q$.
Similarly, set $\app{f}{Q} \defeq \app{f''}{Q}$ for each $\alpha$-path $Q$ in $N''$ not connecting $s''$.
Next consider paths connecting $s'$ in $N'$ and $s''$ in $N''$. 
Observe that $\setprn{s'}$ is a minimum $s'$-cut in $N'$ and $\setprn{s''}$ is a minimum $s''$-cut in $N''$.
An edge $e$ in $N$ joining $X$ and $\tilde{V} \setminus X$ becomes an edge connecting $s'$ in $N'$ and an edge connecting $s''$ in $N''$.
Then $\app{f'}{e} = \app{f''}{e} = \app{c}{e}$ necessarily holds.
Consider an $s'$-path $Q'$ in $N'$ and $s''$-path $Q''$ in $N''$ containing $e$. 
The two paths $Q'$ and $Q''$ are concatenated along $e$ into an $\paren{s_\beta, s_\gamma}$-path $Q$ in $N$ for $s_\beta \in S', s_\gamma \in S''$, and set $\app{f}{Q} \defeq \min \setprn{ \app{f'}{Q'}, \app{f''}{Q''}}$. 
Decrease $f'$ by $\app{f}{Q}$ on $Q'$ (no $s''$-flows in $N''$), and decrease $f''$ by $\app{f}{Q}$ on $Q''$.
Repeating this process until there are no $s'$-flows in $N'$, we obtain a multiflow $f$ in $N$.
Here $f$ is a locking in $N$. 
This follows from the fact (obtained from uncrossing) that for $\alpha \in \intset{k}$ with $s_\alpha \in S'$ (resp. $S''$), a minimum $\alpha$-cut in $N'$ (resp. $N''$) is a minimum $\alpha$-cut in $N$.
Multiflows are kept as node-arc forms.
This procedure, called the \emph{aggregation}, 
can be done in $\Order{nm}$ time as in~\cite[Section 2.2]{Ibaraki1998}.

Suppose that $\absprn{S} \leq 3$.
Suppose that there is an $\alpha$-fringe.
Compute a minimum $\alpha$-cut $X$.
Construct $N'$ and $N''$ as above, find locking multiflows $f'$ in $N'$ and $f''$ in $N''$, and aggregate $f'$ and $f''$ into a locking multiflow $f$ in $N$.
In $N'$, there are two terminals, and a locking multiflow is obtained by a maximum flow.
In $N''$, there are (at most) three terminals but no $\alpha$-fringes.
After recursing at most three times, we arrive at the situation that there are no fringes.
This situation is precisely the same as \cite[Section 2.1]{Ibaraki1998}.
Then a locking multiflow is obtained in at most three max-flow computations.

The time complexity of this algorithm is analyzed in precisely the same way as \cite[Section 2.3]{Ibaraki1998}, sketched as follows.
For simplicity of analysis, we use Orlin's $\Order{nm}$-time algorithm~\cite{Orlin2013} to find a maximum flow and minimum cut.
Let $\app{T}{k, n, m}$ denote the time complexity of the algorithm applied to Potts $k$-submodular functions on graph $\paren{V, E}$ with $\absprn{V} = n$, and $\absprn{E} = m$.
Suppose that the time complexity of the max-flow algorithm and the aggregation procedure are bounded by $D nm$ and by $D'nm$ for constants $D$ and $D'$, respectively. 
We show by induction that $\app{T}{k, n, m} \leq C nm \log k$ for a constant $C$ (to be determined later).
For $k \leq 3$, it holds $\app{T}{3, n, m} \leq \paren{4D + 3D'} nm$. 
Suppose that $k \geq 4$.
Then $\app{T}{k, n, m} \leq \app{T}{k/2, n', m'} + \app{T}{k/2, n'', m''} + D nm + D' nm$ with $n' + n'' = n + 2$. 
By induction, we have
\begin{align*}
  \app{T}{k, n, m} & \leq C n'm' \log k/2 +  C n''m'' \log k/2 + \paren{D+D'} nm \\
                   & \leq C nm \log k - C nm \paren{1 - 2 \paren{\log k/2}/n} + \paren{D+D'}nm \\
                   & \leq C nm \log k - C nm/2 + \paren{D+D'} nm,
\end{align*}
where we use $k \leq n$ and $2 \paren{\log n/2}/n \leq 1/2$.
For $C \geq 4 \paren{D+D'}$, it holds $\app{T}{k, n, m} \leq C nm \log k$ as required.
This completes the proof of Theorem~\ref{thm:Potts}.

\subsection{Enumeration aspect}
\label{sec:enumeration}
The compact representation for $\paren{\sqmeet, \sqjoin}$-closed sets by an elementary PIP is kind of a data compression.
Hence it is natural to consider an efficient way to extract elements of the original $\paren{\sqmeet, \sqjoin}$-closed set.
This corresponds to the enumeration of consistent ideals of an elementary PIP.
As seen in Remark~\ref{lem:pip_and_cnf}, consistent ideals correspond to true assignments of a Boolean 2-CNF.
Thus we can enumerate all consistent ideals in output-polynomial time~\cite{Feder1994} (i.e., the algorithm stops in time polynomial in the length of the input and output).

Maximal consistent ideals are of special interest, as described in Section~\ref{sec:relaxation}. 
For a PIP $P$, let $\maxconsideals{P}$ denote the family of maximal consistent ideals.
Now we consider the enumeration of $\maxconsideals{P}$.
This can also be done in output-polynomial time by using the algorithm of~\cite{Kavvadias2000} in $\Order{k^3 n^3}$ time per output.
We here develop a considerably faster algorithm for the elementary PIP of a Potts $k$-submodular function $\tilde{g}$.
Our algorithm utilizes its amalgamated structure by posets (Theorem~\ref{thm:PIP_Potts}).
In fact, the structure of $\maxconsideals{P}$ is quite simple, which we now explain.
Let $\Sigma_1, \Sigma_2, \ldots, \Sigma_k$ be the posets, and $P \defeq \bigcup_{\alpha \in \intset{k}} \Sigma_\alpha \times \setprn{\alpha}$ the PIP defined in the previous section.
For distinct $\alpha, \beta \in \intset{k}$, let $\Sigma_{\alpha, \beta} \defeq \Sigma_\alpha \cap \Sigma_\beta$ be the subposet of $\Sigma_\alpha$.
In particular, $\Sigma_{\alpha, \beta}$ is equal to $\Sigma_{\beta, \alpha}$ as a set, and the partial order of $\Sigma_{\alpha, \beta}$ is the reverse of that of $\Sigma_{\beta, \alpha}$ by Lemma~\ref{lem:intersection_of_isolating_cuts}. 
Let $\Sigma_{\alpha, 0} \defeq \Sigma_\alpha \setminus \bigcup_{\beta \in \intset{k} \setminus \setprn{\alpha}} \Sigma_{\alpha, \beta}$. 
Now $\Sigma_1 \cup \Sigma_2 \cup \cdots \cup \Sigma_k$ is the disjoint union of $\Sigma_{\alpha, \beta}$ and $\Sigma_{\alpha', 0}$ for $\alpha, \alpha', \beta \in \intset{k}$  with $\alpha < \beta$ (recall Lemma~\ref{lem:uncrossing} that the intersection of three distinct $\Sigma_\alpha, \Sigma_\beta, \Sigma_\gamma$ is empty).
Define the poset $R$ by
\begin{align*}
  R \defeq \bigcup_{1 \leq \alpha < \beta \leq k} \Sigma_{\alpha, \beta},
\end{align*}
where partial order $\leq$ on $R$ is defined as: the relation on $\Sigma_{\alpha, \beta}$ is the same as the partial order of $\Sigma_{\alpha, \beta}$ and there is no relation between $\Sigma_{\alpha, \beta}$ and $\Sigma_{\alpha', \beta'}$ for $\paren{\alpha, \beta} \neq \paren{\alpha', \beta'}$.
For an ideal $J$ of $R$, let $\overline{J}$ be defined by
\begin{align*}
  \overline{J} \defeq
  \paren{
    \bigcup_{1 \leq \alpha \leq k} \Sigma_{\alpha, 0} \times \setprn{\alpha}
  }
  \cup
  \paren{
    \bigcup_{1 \leq \alpha < \beta \leq k}
    \paren{J \cap \Sigma_{\alpha, \beta}} \times \setprn{\alpha}
    \cup
    \paren{\Sigma_{\beta, \alpha} \setminus J} \times \setprn{\beta}
  }.
\end{align*}

\begin{theorem} \label{thm:enumeration_potts}
  The map $J \mapsto \overline{J}$ is a bijection from the ideal family of $R$ to $\maxconsideals{P}$.
\end{theorem}
\begin{proof}
  Let $J$ be an ideal of $R$. 
	We first show that $\overline{J}$ is a consistent ideal. 
	Consider $\paren{X', \alpha} \leq \paren{X, \alpha} \in \overline{J}$.
	If $\paren{X, \alpha} \in \Sigma_{\alpha, 0} \times \setprn{\alpha}$, then $\paren{X', \alpha} \in \Sigma_{\alpha, 0} \times \setprn{\alpha} \subseteq \overline{J}$ since $\Sigma_{\alpha, 0} \times \setprn{\alpha}$ is an ideal by (EP1).
	Suppose that $\paren{X, \alpha} \in \paren{J \cap \Sigma_{\alpha', \beta'}} \times \setprn{\alpha'}$.
	Then $\alpha = \alpha'$.
	Since $J$ is an ideal in $\Sigma_{\alpha, \beta}$, $J \cup \Sigma_{\alpha, 0}$ is an ideal in $\Sigma_\alpha$.
	Consequently $X' \in J \cup \Sigma_{\alpha, 0}$, and $\paren{X', \alpha} \in \overline{J}$.
	Suppose that $\paren{X, \alpha} \in \paren{\Sigma_{\beta', \alpha'} \setminus J} \times \setprn{\beta}$. Then $\alpha = \beta'$.
	Observe that $\Sigma_{\beta', 0} \cup \Sigma_{\beta', \alpha'} \setminus J$ is an ideal in $\Sigma_{\beta'}$. 
	From this, we obtain $\paren{X', \alpha} \in \overline{J}$, as above.
	Since $\overline{J}$ contains exactly one of $\paren{X, \alpha}, \paren{X, \beta} \in \Sigma$ with $\alpha \neq \beta$, the image $\overline{J}$ is consistent and maximal.
	
	Let $I$ be a maximal consistent ideal of $\Sigma$.
	Necessarily $I$ contains $\Sigma_{\alpha, 0} \times \setprn{\alpha}$ for all $\alpha \in \intset{k}$.
	Consider $\paren{X, \alpha}, \paren{X, \beta} \in \Sigma$ with $\alpha \neq \beta$.
	Then $I$ contains exactly one of $\paren{X, \alpha}, \paren{X, \beta}$.
	Otherwise, consider the principal ideal $I'$ of $\paren{X, \alpha}$, and the ideal $I \cup I'$.
	By maximality, $I \cup I'$ is inconsistent. 
	Then there are $\paren{Y, \alpha} \in I'$ and $\paren{Y, \beta'} \in I$ with $\alpha \neq \beta'$.
	By (EP1), (EP2) and Lemma~\ref{lem:intersection_of_isolating_cuts}, there is $\paren{X, \beta'} \in \Sigma$ with $\paren{X, \beta'} \leq \paren{Y, \beta'}$.
	Since $I$ is an ideal, it holds $\paren{X, \beta'} \in I$. 
	Necessarily $\beta = \beta'$ and $\paren{X, \beta} \in I$; this is a contradiction.
	
	For distinct $\alpha, \beta \in \intset{k}$ let $J_{\alpha, \beta} \defeq \setprnsep{X}{\paren{X, \alpha} \in (I \cap \Sigma_{\alpha, \beta}) \times \setprn{\alpha}}$. Then $\Sigma_{\alpha, \beta}$ is the disjoint union of $J_{\alpha, \beta}$ and $J_{\beta, \alpha}$ (as a set).   
	Thus, letting $J \defeq \bigcup_{1 \leq \alpha < \beta \leq k} J_{\alpha, \beta}$,
\end{proof}  

Therefore our problem of enumerating all maximal minimizers of $\tilde g$ is reduced to the enumeration of all ideals of poset $R$.
This is a well-studied enumeration problem. 
One of the current best algorithms is Squire's algorithm~\cite{Squire1995} that enumerates all ideals of an $n$-element poset in amortized $\Order{\log n}$ time per output.
\begin{theorem}
  From the elementary PIP for the minimizer set of Potts $k$-submodular function $\funcdoms{\tilde{g}}{{S_k}^n}{\setR}$, all maximal minimizers of $\tilde{g}$ can be enumerated in amortized $\Order{\log n}$ time per output.  
\end{theorem}

\begin{remark}
  The above poset $R$ may be viewed as a ``compact representation'' of maximal minimizers of Potts $k$-submodular function $\tilde{g}$.
  In a general elementary PIP $P$ (for the minimizer set of a general $k$-submodular function), such a compact representation is still possible if $P$ has a maximal consistent ideal $C$ 
  satisfying the following property: 
  \begin{itemize}
    \item[(P)] $C$ contains exactly one of $x^\circ$ and $y^\circ$ for each $i, j$ of the case (EP2-1).
  \end{itemize}
  In this case, as in $J \mapsto \overline{J}$, there is a bijection between $\consideals{P \setminus C}$ and $\maxconsideals{P}$.
  One can see that the PIP $P \setminus C$ has a simple structure similar to the above poset $R$ (though it is not elementary). 
  We developed an algorithm to enumerate consistent ideals of 
  $P \setminus C$ in $\Order{n}$ time per output, and announced in the conference version of this paper~\cite{isco2016} that such a fast enumeration is possible for maximal consistent ideals 
  of PIP $P$.
  
  However we found an elementary PIP that having no maximal consistent ideal with the property (P); consider PIP $P = \setprn{x, x', y, y', z, z'}$ with $x \mininconsistent x'$, $y \mininconsistent y'$, $z \mininconsistent z'$, $x \succ y' \prec z$, $y \succ z' \prec x$, and 
  $z \succ x' \prec y$.
  Therefore \cite[Theorem 14]{isco2016} is not true for such a PIP.
\end{remark}

\section{Application}
\label{sec:application}

\subsection{$k$-submodular relaxation}
\label{sec:relaxation}

A \emph{$k$-submodular relaxation} $\tilde{f}$ of a function $\funcdoms{f}{\intset{k}^n}{\setRext}$ is a $k$-submodular function on ${S_k}^n$ such that $\app{f}{x} = \app{\tilde{f}}{x}$ for all $x \in \intset{k}^n \paren{\subseteq {S_k}^n}$.
Iwata--Wahlstr\"{o}m--Yoshida~\cite{IwataY2016} investigated $k$-submodular relaxations as a key tool for designing efficient FPT algorithms.
Gridchyn--Kolmogorov~\cite{Gridchyn2013} applied $k$-submodular relaxations to labeling problems on computer vision, which we describe below.

A label assignment is a process of assigning a label to each pixel of a given image.
For example, in the object extraction, each pixel is labeled as ``foreground'' or ``background''.
In stereo matching, the disparity of each pixel is computed from the given two photos taken from slightly different positions, and the pixel is labeled according to the estimated depth.
We consider the labels to be numbered from 1 to $k$.
Such a labeling problem is formulated as the problem of minimizing an \emph{energy function}.
A Potts energy function is simple but widely used energy function.
However the exact minimization of a Potts energy function is computationally intractable.
Gridchyn--Kolmogorov~\cite{Gridchyn2013} applied the $k$-submodular relaxation, that is, the energy function is relaxed to a $k$-submodular function by allowing some pixels to have 0 (meaning ``non-labeled'').
The following property, called \emph{persistency}~\cite{Gridchyn2013, IwataY2016}, is the reason why they introduced the relaxation.

\begin{theorem}[{\cite[Proposition~10]{Gridchyn2013} and~\cite[Lemma~2]{IwataY2016}}] \label{thm:persistency}
  Let $\funcdoms{f}{\intset{k}^n}{\setRext}$ be a function and $\funcdoms{\tilde{f}}{{S_k}^n}{\setRext}$ a $k$-submodular relaxation of $f$.
  For every minimizer $x \in {S_k}^n$ of $\tilde{f}$, there exists a minimizer $y \in \intset{k}^n$ of $f$ such that $x_i \neq 0$ implies $x_i = y_i$ for each $i \in \intset{n}$. 
\end{theorem}

Namely, each minimizer of $\tilde{f}$ gives us partial information about a minimizer of $f$.
An efficient algorithm for minimizing $k$-submodular relaxations of Potts functions was also proposed in~\cite{Gridchyn2013}.
Hence we can obtain a partial labeling extensible to an optimal labeling, which we call a \emph{persistent labeling}.

In Section~\ref{sec:potts} we gave an efficient algorithm to construct the elementary PIP representing all the minimizers of a Potts $k$-submodular function.
Since minimizers that contain more nonzero elements have more information, we want to find a minimizer whose support is largest.
In fact, such minimizers are precisely maximal minimizers.

\begin{proposition}
  Let $M$ be a $\paren{\sqmeet, \sqjoin}$-closed set on ${S_k}^n$.
  The supports of maximal elements in $M$ are the same.
\end{proposition}
\begin{proof}
  Let $x, y \in M$ be maximal and $z \defeq \paren{x \sqjoin y} \sqjoin y$.
  For each $i \in \intset{n}$, it holds $z_i = x_i$ if $y_i = 0$ and $z_i = y_i$ if $y_i \neq 0$.
  In particular, $y \preceq z$ and $\supp z = \supp x \cup \supp y$ hold.
  Since $y$ is maximal, we obtain $y = z$ and $\supp x \subseteq \supp y$.
  By changing the role of $x$ and $y$, we also have $\supp x \supseteq \supp y$.
  Thus $\supp x = \supp y$.
\end{proof}

From this lemma, it turns out that all maximal minimizers of $\tilde{f}$ have the same and largest amount of information about minimizers of $f$.
In labeling problem with Potts energy, all maximal persistent labelings (with respect to a $k$-submodular relaxation) can be efficiently generated by the algorithm in Section~\ref{sec:enumeration}.

\subsection{Experiments}
We implemented our algorithm on the stereo matching problem with Potts energy function~\eqref{eq:potts}.
This has an aspect of the replication of the experiment in~\cite{Gridchyn2013}, but we computed not only one of the persistent labelings but also its PIP-representation.
We used ``tsukuba'' and ``cones'' in the Middlebury data~\cite{Scharstein2003, Scharstein2001} as input images.

\paragraph{Problem setting and formulation.}
We are given photo images $L$ and $R$ taken from left and right positions, respectively.
The images $L$ and $R$ are $N \times M$ arrays such that entries $\sqapp{L}{x, y}$ and $\sqapp{R}{x, y}$ are RGB vectors $\in \setprn{0, 1, 2, \ldots, 255}^3$ of the intensity at pixel $\paren{x, y}$, where each pixel is represented by a pair $\paren{x, y}$ of its horizontal coordinate $x = 1, 2, \ldots, N$ and vertical coordinate $y = 1, 2, \ldots, M$.
The goal of the stereo matching problem is to assign to each pixel the ``disparity label'' $\in \intset{k}$ that represents the depth of the object on the pixel.
We model this problem as a minimization of a Potts energy function~\eqref{eq:potts} on diagonal grid graph $\paren{V, E}$, where $V$ is the set of pixels, and two pixels $\paren{x, y}$ and $\paren{x', y'}$ have an edge in $E$ if and only if $\absprn{x - x'} \leq 1$ and $\absprn{y - y'} \leq 1$.
The first and the second term of~\eqref{eq:potts} are called ``data term'' and ``smoothness term''~\cite{Scharstein2001}, respectively.

For each pixel $i \in V$, the data term $g_i$ measures how well the estimated disparity of pixel $i$ agrees with the pair of given images.
We employed the traditional averaged SSD (sum of squared difference) costs as in~\cite{Gridchyn2013}:
\begin{align}
  \app{g_i}{\alpha} \defeq \text{the nearest integer of $\frac{1}{\absprn{W_i}} \sum_{\paren{x, y} \in W_i} \norm{\sqapp{L}{x, y} - \sqapp{R}{x - d_\alpha, y}}^2$}
  \quad
  \paren{\alpha \in \intset{k}},
\end{align}
where $W_i$ is the $9 \times 9$ window centered at $i = \paren{x, y}$ (i.e., the set of pixels $\paren{x', y'}$ with $\absprn{x' - x} \leq 4$ and $\absprn{y' - y} \leq 4$), $\norm{\cdot}$ is the 2-norm, $d_\alpha \geq 0$ is the disparity corresponding to the label $\alpha$.
We set $d_\alpha \defeq 2 \paren{\alpha - 1}$ for each $\alpha \in \intset{k}$.

For each pair of adjacent pixels $\setprn{i, j} \in E$, the smoothness term increases the energy by $\lambda_{i, j}$ if $i$ and $j$ have different labels.
We set every $\lambda_{i, j}$ to be the same value $\lambda$ as in~\cite{Gridchyn2013}, and conducted experiments with $\lambda = 1$ and 20 to see the effect of $\lambda$.

\paragraph{Experimental results.}

\begin{figure}[tbp]
  \begin{minipage}[t][4.5cm][t]{0.33\linewidth}
    \centering
    \includegraphics[scale=0.35]{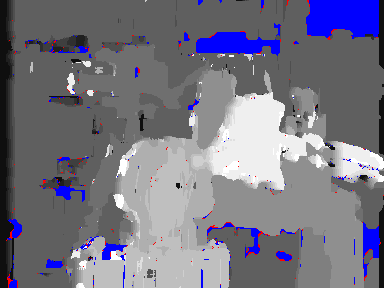}
    \subcaption{$\lambda = 1$}
  \end{minipage}%
  \begin{minipage}[t][4.5cm][t]{0.33\linewidth}
    \centering
    \includegraphics[scale=0.35]{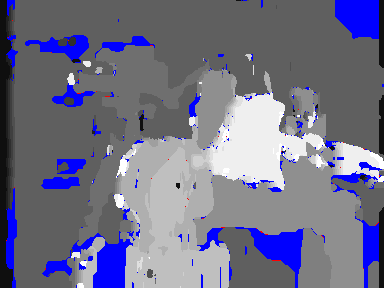}
    \subcaption{$\lambda = 20$}
  \end{minipage}%
  \begin{minipage}[t][4.5cm][t]{0.33\linewidth}
    \centering
    \includegraphics[scale=0.35]{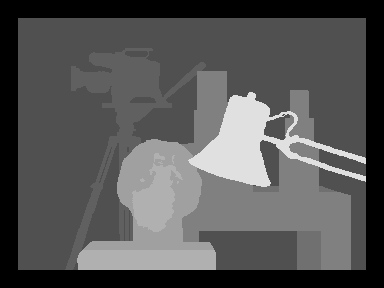}
    \subcaption{ground truth}
  \end{minipage}
  \begin{minipage}[b]{0.33\linewidth}
    \centering
    \includegraphics[scale=0.3]{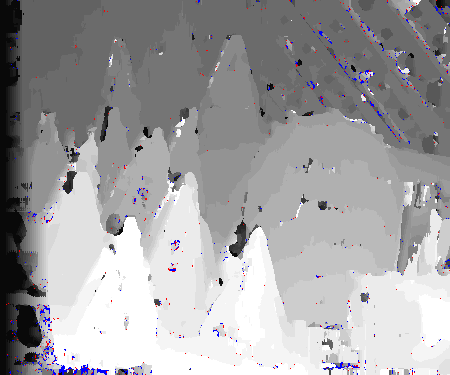}
    \subcaption{$\lambda = 1$}
  \end{minipage}%
  \begin{minipage}[b]{0.33\linewidth}
    \centering
    \includegraphics[scale=0.3]{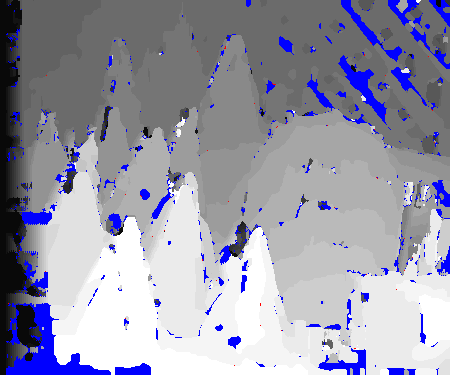}
    \subcaption{$\lambda = 20$}
  \end{minipage}%
  \begin{minipage}[b]{0.33\linewidth}
    \centering
    \includegraphics[scale=0.3]{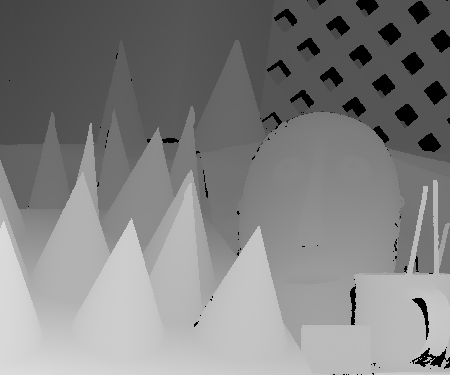}
    \subcaption{ground truth}
  \end{minipage}
  \caption{Results for ``tsukuba'' (top row) with $k = 16$ and ``cones'' (bottom row) with $k = 26$.}
  \label{fig:result_image}
\end{figure}

Figure~\ref{fig:result_image} and Table~\ref{tbl:result} show the results of our experiment.
The pixels labeled in the minimum persistent labeling are colored in gray of the brightness corresponding to each label.
The blue pixels are unlabeled even in maximal persistent labelings.
The red pixels are the difference between the minimum persistent labeling and maximal ones, i.e., the pixel was labeled in (any of) maximal persistent labelings but not in the minimum one. 
We can observe that there are few red pixels as mentioned in~\cite{Gridchyn2013}, and they are mainly located on the boundary of two regions with different labels.
A possible reason is the following: consider a simple 1-dimensional case where a pixel $i$ is adjacent only to pixels $j$ and $j'$.
Let $x \in {S_k}^n$ be the minimum persistent labeling and assume $x_{j'} \neq x_j \neq 0$.
Then the increment of the energy is the same ($= \lambda$) even if $x_i$ is set to any of $\setprn{0, x_j, x_{j'}}$.
Therefore the pixel $i$ will be red if $\app{\tilde{g}_i}{0} = \app{\tilde{g}_i}{x_j}$, and thus we think that this will occur in boundaries more frequently than inside of regions.

With regard to the effect of $\lambda$, the larger $\lambda$ decreases the percentages of gray and red pixels on both tsukuba and cones, and increases the blue pixels to the contrary.
This result agrees with the experiments in~\cite{Gridchyn2013}.
We consider that this is due to the fact that the value of each $\app{\tilde{g}_i}{0}$ is moderately lower in $\tilde{g}_i$ since $\app{\tilde{g}_i}{0}$ is the average of the minimum and the second minimum values of $g_i$ as described in Section~\ref{sec:potts}.
Hence if $\lambda$ is large, the energy will be lower just by letting all $x_i \defeq 0$ than by tuning each $x_i$ finely according to the values of the corresponding data term $\tilde{g}_i$.

\begin{table}[tbp]
  \centering
  \caption{Experimental result}
  \begin{tabular}{l|r|ccc|c}\hline
    image   & \multicolumn{1}{|c|}{$\lambda$} & \% of gray & \% of red & \% of blue & \# of max. pers. labelings \\ \hline \hline
    tsukuba & $1$  & 93.84 & 0.53 & 5.63 & $2^{66} \times 3^{5} \times 4 \times 5^2 \times 18$\\
    tsukuba & $20$ & 90.64 & 0.07 & 9.29 & $2^{11}$ \\
    cones   & $1$  & 99.00 & 0.30 & 0.70 & $2^{114} \times 3^{9} \times 5$ \\
    cones   & $20$ & 93.37 & 0.04 & 6.59 & $2^{17} \times 3$ \\
    \hline
  \end{tabular}
  \label{tbl:result}
\end{table}

\begin{table}[tbp]
  \centering
  \caption{Experimental result without rounding in $g_i$}
  \begin{tabular}{l|r|ccc|p{50pt}c}\hline
    \multirow{2}{*}{image} & \multicolumn{1}{c|}{\multirow{2}{*}{$\lambda$}} & \multirow{2}{*}{\% of gray} & \multirow{2}{*}{\% of red} & \multirow{2}{*}{\% of blue} & \multicolumn{2}{c}{\# of pers. labelings} \\ \cline{6-7}
    &&&&& \hfil all \hfil & \hfil max \hfil \\ \hline \hline
    tsukuba & $1$  & 94.18 & 0.007 & 5.81 & \hfil $2^4 \times 3^3$ \hfil & \hfil $2^3$ \hfil \\
    tsukuba & $20$ & 90.71 & 0.002 & 9.29 & \hfil $2^2$            \hfil & \hfil 1     \hfil \\
    cones   & $1$  & 99.20 & 0.007 & 0.80 & \hfil $2^7 \times 3^2$ \hfil & \hfil $2^2$ \hfil \\
    cones   & $20$ & 93.40 & 0.000 & 6.60 & \hfil 1                \hfil & \hfil 1     \hfil \\
    \hline
  \end{tabular}
  \label{tbl:result_without_rounding}
\end{table}

\paragraph{The structure of the PIP.}

\begin{figure}[tbp]
  \centering
  \includegraphics[scale=0.16]{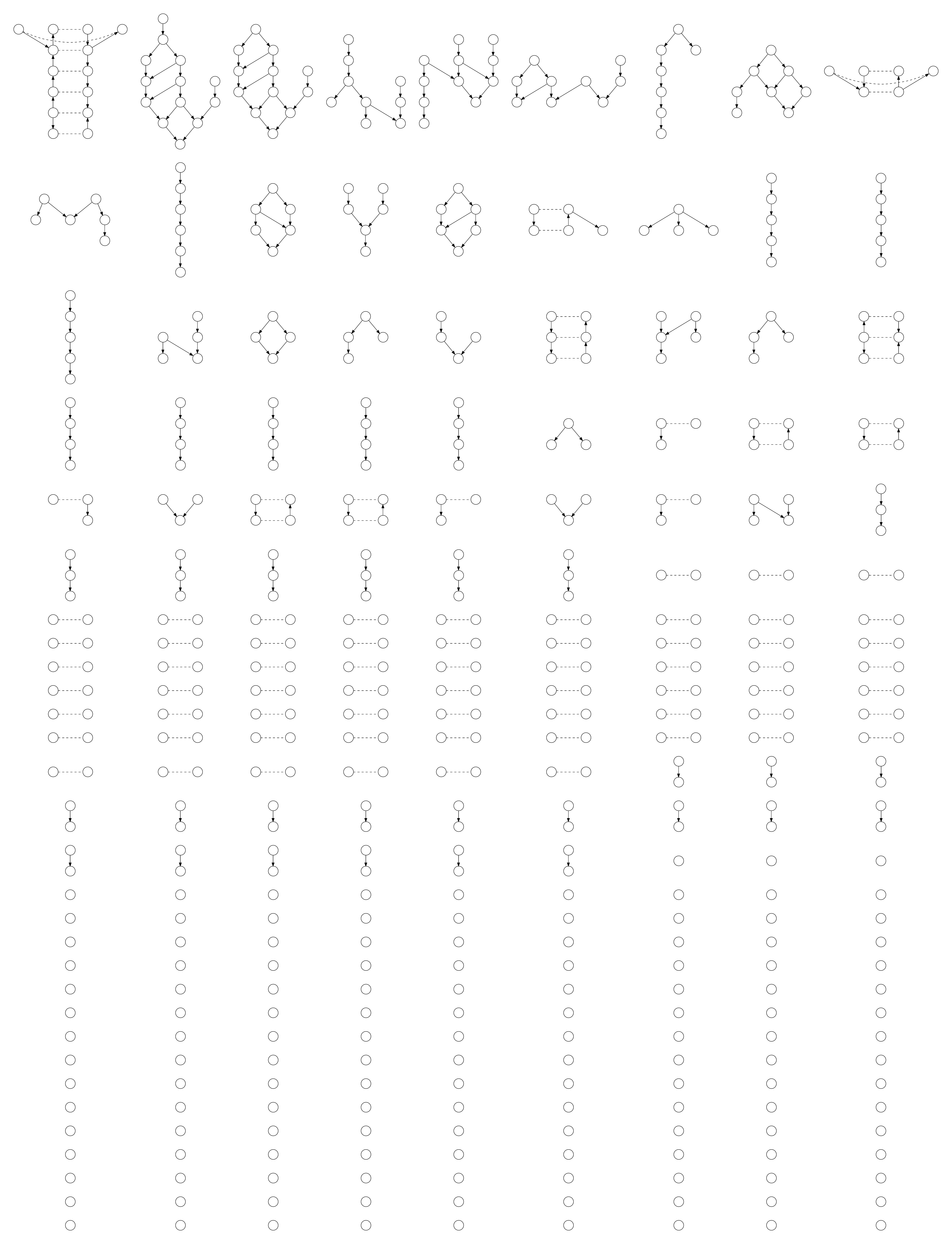}
  \caption{The elementary PIP for persistent labelings of tsukuba with $\lambda = 1$.}
  \label{fig:result_pip}
\end{figure} 

Figure~\ref{fig:result_pip} shows the PIP-representation for the partial labelings on tsukuba with $\lambda = 1$.
The PIP consists of many small connected PIPs, which correspond to each of connected red regions in Figure~\ref{fig:result_image}.
Thus we can easily calculate the number of maximal persistent labelings by multiplying the number of maximal consistent ideals of each connected PIP (notice that an ideal of the PIP is maximal and consistent if and only if it contains all elements having no inconsistent pair and one element of each minimally inconsistent pairs).
The right most column of Table~\ref{tbl:result} shows the number of maximal persistent labelings in each experiment.
We discovered the fact that there are plenty of maximal persistent labelings even though the percentages of red pixels are small.

\paragraph{Effect of rounding.}
In our experiment, the data term $g_i$ is defined to be integer-valued by rounding a rational to the nearest integer.
One of the referees conjectured that if $g_i$ is defined without the rounding, then there is a unique persistent labeling.
We did an experiment to verify this conjecture.
Table~\ref{tbl:result_without_rounding} shows the result.
Without the rounding, the percentage of the red pixels  and the number of persistent labelings dramatically decrease, though there are some cases where a persistent labeling is not unique.

\section*{Acknowledgments}
We thank Kazuo Murota, Satoru Fujishige, and the referees for helpful comments.
This work was partially supported by JSPS KAKENHI Grant Numbers 25280004, 26330023, 26280004, and by JST, ERATO, Kawarabayashi Large Graph Project.

\end{document}